\documentclass[preprint,12pt]{elsarticle}

\usepackage{amssymb}

\usepackage{textcomp}
\usepackage{amsmath}
\usepackage{hyperref}
\usepackage[textheight=20cm, textwidth=14cm]{geometry}

\usepackage{chngcntr}

\newtheorem{theorem}{Theorem}

\newtheorem{lemma}{Lemma}
\newtheorem{corollary}{Corollary}

\newtheorem{problem}{Problem}

\newproof{proof}{Proof}

\newdefinition{remark}{Remark}

\numberwithin{equation}{section}

\newcommand\restr[2]{{
		\left.\kern-\nulldelimiterspace 
		#1 
		\vphantom{\big|} 
		\right|_{#2} 
}}

\journal{Journal of Mathematical Analysis and Applications}

\begin{document}
	
\begin{frontmatter}

\title{Frequency theorem for parabolic equations and its relation to inertial manifolds theory}

\author{Mikhail Anikushin}
\address{Department of
	Applied Cybernetics, Faculty of Mathematics and Mechanics,
	Saint Petersburg State University, 28 Universitetskiy prospekt, Peterhof, St. Petersburg 198504, Russia}
\address{Euler International Mathematical Institute, St. Petersburg Department of Steklov Mathematical Institute of Russian Academy of Sciences, 27 Fontanka, St. Petersburg 191011, Russia}
\ead{demolishka@gmail.com}
\date{\today}



\begin{abstract}
We obtain a version of the Frequency Theorem (a theorem on solvability of certain operator inequalities), which allows to construct quadratic Lyapunov functionals for semilinear parabolic equations. We show that the well-known Spectral Gap Condition, which was used in the theory of inertial manifolds by C.~Foias, R.~Temam and G.~R.~Sell, is a particular case of some frequency inequality, which arises within the Frequency Theorem. In particular, this allows to construct inertial manifolds for semilinear parabolic equations (including also some non-autonomous problems) in the context of a more general geometric theory developed in our adjacent works. This theory is based on quadratic Lyapunov functionals and generalizes the frequency-domain approach used by R.~A.~Smith. We also discuss the optimality of frequency inequalities and its relationship with known old and recent results in the field.
\end{abstract}


\end{frontmatter}


\section{Introduction}
We start from precise statements of some of our main results. In this introduction we assume that all the vector spaces are complex. Let $A$ be the generator of a $C_{0}$-semigroup $G(t)$, where $t \geq 0$, acting in a Hilbert space $\mathbb{H}$. For the theory of $C_{0}$-semigroups we recommend the monographs of S.~G.~Krein \cite{Krein1971} or K-J.~Engel and R.~Nagel \cite{EngelNagel2000}. Let $\mathcal{D}(A) \subset \mathbb{H}$ denote the domain of $A$. Let $\Xi$ be another Hilbert space and $B \in \mathcal{L}(\Xi;\mathbb{H})$ \footnote{Here and further $\mathcal{L}(\mathbb{E};\mathbb{F})$ denotes the space of bounded linear operators from a Banach space $\mathbb{E}$ to a Banach space $\mathbb{F}$. If $\mathbb{E}=\mathbb{F}$, we usually write $\mathcal{L}(\mathbb{E})$.}. We consider the abstract evolution equation in $\mathbb{H}$
\begin{equation}
\label{EQ: ControlSystem}
\dot{v}(t) = Av(t) + B\xi(t)
\end{equation}
which is called a \textit{control system}. It is well-known that for every $T>0$, $v_{0} \in \mathbb{H}$ and $\xi=\xi(\cdot) \in L_{2}(0,T;\Xi)$ there exists a unique mild solution $v(t)=v(t, v_{0},\xi)$, where $v(0)=v_{0}$ and $t \in [0,T]$, to \eqref{EQ: ControlSystem}, which is a continuous $\mathbb{H}$-valued function and given by
\begin{equation}
\label{EQ: MildSolution}
v(t) = G(t)v_{0} + \int_{0}^{t}G(t-s)B\xi(s)ds.
\end{equation}

Let $\mathbb{H}_{\alpha}$ be a Hilbert space, which is continuously and densely embedded in $\mathbb{H}$. We also suppose that under the embedding we have $\mathcal{D}(A) \subset \mathbb{H}_{\alpha}$ and, consequently, the latter embedding is continuous by the closed graph theorem. We denote the norms in $\mathbb{H}$ and $\mathbb{H}_{\alpha}$ by $|\cdot|$ and $|\cdot|_{\alpha}$ respectively and the corresponding inner products by $(\cdot,\cdot)$ and $(\cdot,\cdot)_{\alpha}$. One may think of $\mathbb{H}_{\alpha}$ as a space in the scale of Hilbert spaces given by powers of a self-adjoint positive-definite operator (or a sectorial operator) in $\mathbb{H}$ and $A$ being a perturbation of such an operator (see Section \ref{SEC: IMParabSemilinear}). This justifies the notation for such a space. Let us consider the following condition.

\begin{description}
	\item[\textbf{(RES)}] The operator $A$ does not have spectrum on the imaginary axis and the operators $(A-i\omega I)^{-1}$ are bounded in the norm of $\mathcal{L}(\mathbb{H};\mathbb{H}_{\alpha})$ uniformly for all $\omega \in \mathbb{R}$.
\end{description}

Now let $\mathcal{F}(v,\xi)$, where $v \in \mathbb{H}_{\alpha}$ and $\xi \in \Xi$ be a Hermitian form such as
\begin{equation}
\label{EQ: QuadraticFormComplex}
\mathcal{F}(v,\xi) = (\mathcal{F}_{1}v,v)_{\alpha} + 2\operatorname{Re}(\mathcal{F}_{2}v,\xi)_{\Xi} + (\mathcal{F}_{3}\xi,\xi)_{\Xi},
\end{equation}
where $(\cdot,\cdot)_{\Xi}$ denotes the inner product in $\Xi$ and $\mathcal{F}^{*}_{1}=\mathcal{F}_{1} \in \mathcal{L}(\mathbb{H}_{\alpha})$, $\mathcal{F}_{2} \in \mathcal{L}(\mathbb{E};\Xi)$, $\mathcal{F}_{3}^{*}=\mathcal{F}_{3} \in \mathcal{L}(\Xi)$. In Section \ref{SEC: IMParabSemilinear} such a form will be used to determine a class of nonlinearities.

The following version of the Frequency Theorem is one of the main results of this paper. 

\begin{theorem}
	\label{TH: ParabComplexFreqTheorem}
	Let $A$ be also the generator of a $C_{0}$-semigroup $\mathbb{H}_{\alpha}$ and let \textbf{(RES)} be satisfied. Then the following conditions for the pair $(A,B)$ and the form $\mathcal{F}$ are equivalent:
	
	1. For some $\delta'>0$ we have $\mathcal{F}(-(A - i\omega I)B\xi,\xi) \leq -\delta' |\xi|^{2}_{\Xi}$ for all $\xi \in \Xi$ and $\omega \in \mathbb{R}$.
	
	2. There are $\delta>0$ and $P \in \mathcal{L}(\mathbb{H}_{\alpha})$, which is self-adjoint in $\mathbb{H}_{\alpha}$ and such that for $V(v):=( Pv, v )_{\alpha}$ and any $T \geq 0$ we have
	\begin{equation}
	V(v(T)) - V(v_{0}) + \int_{0}^{T} \mathcal{F}(v(s),\xi(s))ds \leq -\delta \int_{0}^{T}(|v(s)|^{2}_{\alpha}+|\xi(s)|^{2}_{\Xi})ds,
	\end{equation}
	where $v(t)=v(t,v_{0},\xi)$ is the solution to \eqref{EQ: ControlSystem} with arbitrary $v(0)=v_{0} \in \mathbb{H}_{\alpha}$ and $\xi(\cdot) \in L_{2}(0,T;\Xi)$ such that $v(\cdot) \in C([0,T];\mathbb{H}_{\alpha})$.
\end{theorem}

\noindent The condition contained in item 1 of the theorem is called \textit{frequency-domain condition} or \textit{frequency inequality}. In Section \ref{SEC: IMParabSemilinear} we discuss applications of Theorem \ref{TH: ParabComplexFreqTheorem} (or its real version, Theorem \ref{TH: ParabRealFreqTheorem}) and some adjacent to it results to semilinear parabolic equations. In particular, we show that the well-known Spectral Gap Condition used in the theory of inertial manifolds is a special case of some frequency-domain condition.

Speaking in terms of Section \ref{SEC: IMParabSemilinear}, previously existing versions of the Frequency Theorem proved by A.~L.~Liktarnikov and V.~A.Yakubovich allow to study the considered class of equations only for $\alpha=0$ (by results of \cite{Likhtarnikov1977}) and $\alpha=1/2$ (by results of \cite{Likhtarnikov1976}). Moreover, these papers contain some additional assumptions, which can be relaxed. Namely, it was recently discovered by A.~V.~Proskurnikov \cite{Proskurnikov2015} that the controllability (or stabilizability) assumption used in \cite{Likhtarnikov1977,Likhtarnikov1976} can be relaxed if one needs only the operator $P$, but not the correctness of some optimization problem. These ideas were extended for the Frequency Theorem from \cite{Anikushin2020FreqDelay}, which covers delay equations, and Theorem \ref{TH: ParabComplexFreqTheorem} also goes in this direction. This relaxation and new versions of the Frequency Theorem allowed the present author to unify and generalize many results of R.~A.~Smith (for example, \cite{Smith1994PB1,Smith1994PB2}) on the Poincar\'{e}-Bendixson theory \cite{Anik2020PB}, convergence theorems \cite{Anikushin2021DiffJ, Anikushin2020Geom} and inertial manifolds \cite{Anikushin2020Red,Anikushin2020Semigroups,Anikushin2020Geom}. Now it became clear that R.~A.~Smith was dealing with a potentially more general (than the one developed by C.~Foias, G.~R.~Sell and R.~Temam \cite{FoiasSellTemam1988}) theory of inertial manifolds. The present paper along with other works of the author reveal this (see Section \ref{SEC: IMParabSemilinear}). 

Armed with proper quadratic Lyapunov functionals, we can establish well-known (see \cite{Zelik2014,KostiankoZelikSA2020,Miclavcic1991,ChowLuSell1992,SellYou1992}) properties of inertial manifolds such as Lipschitzity, exponential tracking, $C^{1}$-differentiability and normal hyperbolicity from integral geometric arguments, i.~e. the arguments that do not use infinitesimal generators (=differential equations). This makes the abstract theory applicable also to (parabolic) delay equations \cite{Anikushin2020FreqDelay} and parabolic equations with nonlinear boundary conditions \cite{Likhtarnikov1976}. Some of these results are contained in \cite{Anikushin2020Red,Anikushin2020Semigroups} and a complete theory is presented in \cite{Anikushin2020Geom}. However, without integral operators and fixed points, there is a price to pay, which in the case of \cite{Anikushin2020Geom} is the compactness assumption. Note that Theorem \ref{TH: ParabComplexFreqTheorem}, when $\mathbb{H}_{\alpha} = \mathbb{H}$, can be applied to study hyperbolic equations also. But we cannot construct inertial manifolds by means of our theory from \cite{Anikushin2020Geom} due to the compactness assumption. Thus, it is interesting how to relax this condition. For example, in the work of N.~Koksch and S.~Siegmund \cite{KokschSiegmund2002}, instead of compactness, it is used some conditions imposed on projectors.

In the recent work of A. Kostianko et al. \cite{KostiankoZelikSA2020} (see also the survey of S.~Zelik \cite{Zelik2014}) it is shown that the Spatial Averaging Principle suggested by J.~Mallet-Paret and G.~R.~Sell \cite{MalletParetSell1988IM} (used in order to construct inertial manifolds for certain scalar parabolic equations in 2D and 3D domains, where the Spectral Gap Condition is not satisfied) also leads to the existence of certain quadratic Lyapunov functionals with similar (in fact, potentially more general) properties as \textbf{(H3)} from our Theorem \ref{TH: ParabEqsMainAssumptionsVer} in Section \ref{SEC: IMParabSemilinear}. It is interesting is there some connection between the Frequency Theorem and the Spatial Avering Principle. It is known that analogs of the frequency inequality for non-stationary optimization problems (where $A$, $B$ and $\mathcal{F}$ may depend on time) can be given in terms of exponential dichotomies for a certain Hamiltonian system associated with the optimization problem (see R.~Fabbri, R.~Johnson and C.~N\'{u}\~{n}ez \cite{Fabbri2003FreqTh} for finite-dimensional problems and uniform exponential dichotomies). From this a conjecture can be stated that the conditions of the Spatial Averaging Principle imply some kind of nonuniform exponential dichotomy for a certain Hamiltonian system associated with some optimization problem posed for the linearized (variational) system. So, we hope that this and our adjacent works will stimulate developments of the Frequency Theorem for infinite-dimensional non-stationary problems based on nonuniform exponential dichotomies.

To provide understanding in the theory of inertial manifolds and our approach, we have to place emphasis on the problem of recovering semi-dichoto-\ mies after perturbations and consider it as the main problem in the theory. This means that we have to construct not only the inertial manifold, but also the stable foliation. Besides our paper \cite{Anikushin2020Geom}, we know only the work of R.~Rosa and R.~Temam \cite{RosaTemam1996}, where such a foliation is constructed. So, for any point from the phase space there is a unique tracking trajectory on the inertial manifold. However, this uniqueness went unnoticed in \cite{RosaTemam1996} and, again to the best of our knowledge, it was firstly discovered by M.~Miklav\v{c}i\v{c} \cite{Miclavcic1991}, who, however, did not proceed to construct the foliation explicitly. So, this problem did not get much attention even from the authors that was close to it. Nevertheless, at this point quadratic Lyapunov functionals naturally arise to characterize semi-dichotomies. Thus, natural conditions for the existence of inertial manifolds are not the separate conditions such as the Cone Condition and the Squeezing Property used in most of works, but the squeezing w.~r.~t. quadratic Lyapunov functionals (which contain both the Cone Condition and the Squeezing Property). The role of quadratic Lyapunov functionals in the perturbation theory is well-known and the Frequency Theorem only promotes this role. This may also explain why the generalization of inertial manifolds theory for Banach spaces presented by N.~Koksch and S.~Siegmund in \cite{KokschSiegmund2002}, which is based on too straight generalizations of the classical conditions, did not succeed in applications.

From the very beginning of developments on frequency-domain methods, there appeared an analogy between two approaches, which led to the same frequency conditions. The first one (originated by V.~A.~Yakubovich) uses various classes of (not necessarily quadratic) Lyapunov functionals, which can be constructed from the Frequency Theorem. The second one, which was originated by V.~M.~Popov \cite{Popov1961}, is based on a priori estimates (which also highly rely on the Fourier transform at key moments) for integrals and does not appeal to Lyapunov functionals. Since Lyapunov functionals provide geometric insight into all the constructions and unify many scattered results in the field, they are more preferable for us.

Thus, from the above point of view, the result of M.~Miklav\v{c}i\v{c} \cite{Miclavcic1991} follows the line of V.~M.~Popov and it is not surprising that it leads to the same conditions, when both approaches work (see Subsection \ref{SUBSEC: ParabOtherFreqIneq}). It should be also noted that R.~A.~Smith used quadratic functionals for ODEs, but abandoned this approach (in favor of the method of a priori integral estimates) for infinite-dimensional problems \cite{Smith1994PB1,Smith1994PB2}, being unable to show the existence of such functionals.

In most of works classes of considered quadratic functionals are limited to a very special class (let us call it), which is suitable mostly for self-adjoint problems (see, for example, the survey of S.~Zelik \cite{Zelik2014}). Therefore, the corresponding proofs rely on this specificity at some moments. Quadratic functionals constructed by the Frequency Theorem are given by dynamics and they are usually non-standard (for example, for reaction-diffusion equations the Frequency Theorem constructs a compact operator $P$ \cite{Anikushin2019+OnCom}). It is also impossible to explicitly write such functionals or their analogs appropriate for non-self-adjoint problems (in finite-dimensions there are algorithms for approximating such functionals).

In applications, the Frequency Theorem often provides optimal and flexible conditions, which lead to sharp estimates for the rate of attraction and the dimension of inertial manifolds. In Subsection \ref{SUBSEC: OptimalityParab} we discuss in what sense the optimality should be understood. Namely, we will show that the frequency inequality is equivalent to the existence of a common quadratic Lyapunov functional for the class of problems (see Theorem \ref{TH: ParabOptimalityTh}). We show that the classical self-adjoint optimality result (firstly obtained by M.~Miklav\v{c}i\v{c} \cite{Miclavcic1991} and A.~V.~Romanov \cite{Romanov1994}) and some recent results of A.~Kostianko and S.~Zelik \cite{KostiankoZelik2019Kwak} and V.~V.~Chepyzhov, A.~Kostianko and S.~Zelik \cite{ChepyzhovKostiankoZelik2019} for certain non-self-adjoint problems can be considered as the result that shows the necessity of the frequency inequality for the existence of semi-dichotomies in the class of linear perturbations. We also discuss why it is almost impossible to obtain sharper conditions in the case when the frequency inequality is violated, but we still have semi-dichotomies for linear perturbations, by means of counterexamples to the Kalman conjecture (see the survey of G.~A.~Leonov and N.~V.~Kuznetsov \cite{LeoKuz2013Hidden}) from stability theory (see Remark \ref{REM: UniformAndKalman}).

It should be noted that the Frequency Theorem from \cite{Likhtarnikov1977} was rediscovered by J.-Cl.~Louis and D.~Wexler \cite{LouisWexler1991}. It may seem interesting that for various linear problems given by delay or parabolic equations, the exponential stability can be established with the use of a bounded in $\mathbb{H}$ operator $P$ as a theorem of R.~Datko states \cite{Datko1970}. But it seems natural that for nonlinear problems this is not so and one needs further restrictions as $P \in \mathcal{L}(\mathbb{H}_{\alpha})$ in our case (or $P \in \mathcal{L}(\mathbb{E};\mathbb{E}^{*})$ in the case of delay equations \cite{Anikushin2020FreqDelay}). However, certain nonlinear problems (mainly of parabolic type) may posses a compact or even a Hilbert–Schmidt operator $P$ (see \cite{Anikushin2019+OnCom}).

The Frequency Theorem in finite-dimensions, known as the Kalman-Yaku-\ bovich-Popov lemma (KYP lemma), has great success in the study of stability, periodic and almost periodic solutions and dimensional-like properties of autonomous and non-autonomous ODEs (see the monographs of A.~Kh.~Gelig, G.~A.~Leonov and V.~A.~Yakubovich \cite{Gelig1978}; N.~V.~Kuznetsov and V.~Reitmann \cite{KuzReit2020}).

This paper is organized as follows. In Section \ref{SEC: ProofsFreqTheorem} we prove a version of the Frequency Theorem for the quadratic regulator problem, which requires controllability assumptions. In Section \ref{SEC: RelaxingControllability} we prove Theorem \ref{TH: ParabComplexFreqTheorem}, which relaxes these assumptions in the case of our interest. In Section \ref{SEC: Realification} we consider a real version of Theorem \ref{TH: ParabComplexFreqTheorem} and discuss sign properties of the corresponding operator. In Section \ref{SEC: IMParabSemilinear} we apply our results to study inertial manifolds for semilinear parabolic equations. Here we discuss in what sense the frequency inequality is optimal and consider some of its forms, including the Spectral Gap Condition, R.~A.~Smith's condition and the Circle Criterion. At the end we present an illustrative example.
\section{Optimization of quadratic functionals}
\label{SEC: ProofsFreqTheorem}

As in the introduction within this section we suppose that all the spaces $\mathbb{H}$, $\mathbb{H}_{\alpha}$ and $\Xi$ are complex and the Hermitian form $\mathcal{F}$ is given in the same way as in \eqref{EQ: QuadraticFormComplex}.

We will also use some kind of $L_{2}$-controllability assumption w.~r.~t. initial conditions from $\mathbb{H}_{\alpha}$ as follows.

\begin{description}
	\item[\textbf{(CONT)}] For every $v_{0} \in \mathbb{H}_{\alpha}$ there exists $\xi(\cdot) \in L_{2}(0,+\infty;\Xi)$ such that for $v(\cdot)=v(\cdot,v_{0},\xi)$ we have $v(\cdot) \in L_{2}(0,+\infty;\mathbb{H}_{\alpha})$.
\end{description}

Now let us consider the space $\mathcal{Z}= L_{2}(0,+\infty;\mathbb{H}_{\alpha}) \times L_{2}(0,+\infty;\Xi)$. Let $\mathfrak{M}_{v_{0}}$ be the set of all $( v(\cdot),\xi(\cdot) ) \in \mathcal{Z}$ such that $v(\cdot)=v(\cdot,v_{0},\xi)$. Every such pair is called a \textit{process through} $v_{0}$. Our assumption \textbf{(CONT)} means that $\mathfrak{M}_{v_{0}}$ is non-empty for all $v_{0} \in \mathbb{H}_{\alpha}$. If this is so, from \eqref{EQ: MildSolution} one can see that $\mathfrak{M}_{v_{0}}$, where $v_{0} \in \mathbb{H}_{\alpha}$, is a closed affine subspace of $\mathcal{Z}$ given by a proper translate of $\mathfrak{M}_{0}$. Moreover, we have the following lemma.

\begin{lemma}
	\label{LEM: OperatorSubspaceLemma}
	There exists $D \in \mathcal{L}(\mathbb{H}_{\alpha};\mathcal{Z})$ such that $\mathfrak{M}_{v_{0}} = \mathfrak{M}_{0} + Dv_{0}$.
\end{lemma}
\begin{proof}
    For a fixed element $v_{0} \in \mathbb{H}_{\alpha}$ let $z \in \mathfrak{M}_{v_{0}}$ be any process through $v_{0}$. Since the difference of any two processes from $\mathfrak{M}_{v_{0}}$ is a process from $\mathfrak{M}_{0}$, we have that the definition $Dv_{0}:=z^{\bot}$, where $z^{\bot}$ is the part of $z$ orthogonal (in $\mathcal{Z}$) to $\mathfrak{M}_{0}$, is correct, i. e. it does not depend on $z \in \mathfrak{M}_{v_{0}}$. It is also clear that $z^{\bot} \in \mathfrak{M}_{v_{0}}$. Since $A$ is the generator of a $C_{0}$-semigroup in $\mathbb{H}$, using \eqref{EQ: MildSolution} it is straightforward to verify that $D$ is closed and, consequently, by the closed graph theorem, $D \in \mathcal{L}(\mathbb{H}_{\alpha};\mathcal{Z})$ as it is required.
\end{proof}

Let us consider the quadratic functional $\mathcal{J}_{\mathcal{F}}(v(\cdot),\xi(\cdot))$ in the space $\mathcal{Z}$ given by
\begin{equation}
\label{EQ: QuadraticFunctional}
\mathcal{J}_{\mathcal{F}}(v(\cdot),\xi(\cdot)) := \int_{0}^{\infty}\mathcal{F}(v(t),\xi(t))dt
\end{equation}
and the problem of minimization of $\mathcal{J}_{\mathcal{F}}$ on affine subspaces $\mathfrak{M}_{v_{0}}$ for $v_{0} \in \mathbb{H}_{\alpha}$. Under \textbf{(CONT)} this problem is correct. We say that a process $z \in \mathfrak{M}_{v_{0}}$ is \textit{optimal} if it is the minimum point of $\mathcal{J}_{\mathcal{F}}$ on $\mathfrak{M}_{v_{0}}$.

Let us also put $\mathcal{Z}_{1}:=L_{2}(0,+\infty;\mathbb{H}_{\alpha})$ and $\mathcal{Z}_{2}:=L_{2}(0,+\infty;\Xi)$ and consider the value
\begin{equation}
\label{EQ: Alpha1}
\alpha_{1}:= \inf_{(v(\cdot),\xi(\cdot)) \in \mathfrak{M}_{0}} \frac{\mathcal{J}_{\mathcal{F}}(v(\cdot),\xi(\cdot))}{\|v(\cdot)\|^{2}_{\mathcal{Z}_{1}} + \|\xi\|_{\mathcal{Z}_{2}}^{2}}.
\end{equation}
and also the value
\begin{equation}
\label{EQ: Alpha2}
\alpha_{2} := \inf \frac{ \mathcal{F}(v,\xi) }{ |v|^{2}_{\alpha} + |\xi|^{2}_{\Xi} },
\end{equation}
where the infimum is taken over all $\omega \in \mathbb{R}$, $v \in \mathcal{D}(A)$ and $\xi \in \Xi$ such that $i\omega v = Av + B\xi$. Moreover, we also consider the value
\begin{equation}
\label{EQ: Alpha3}
\alpha_{3} := \inf_{\omega \in \mathbb{R}} \inf_{ \xi \in \Xi} \frac{\mathcal{F}(-(A-i\omega I)^{-1}B\xi,\xi) }{|\xi|^{2}_{\Xi}}
\end{equation}
and say that it is well-defined if the spectrum of $A$ does not intersect with the imaginary axis and the operators $(A-i\omega I)^{-1}B$ are bounded in the norm of $\mathcal{L}(\Xi,\mathbb{H}_{\alpha})$ uniformly for all $\omega \in \mathbb{R}$.

\begin{lemma}
	\label{LEM: ParabOptimizationOnAffineSubspaces}
	Suppose $\alpha_{1} > 0$ and \textbf{(CONT)} is satisfied. Then for every $v_{0} \in \mathbb{H}_{\alpha}$ there exists a unique minimum $(v^{0}(\cdot,v_{0}), \xi^{0}(\cdot,v_{0}))$ of $\mathcal{J}_{\mathcal{F}}$ on $\mathfrak{M}_{v_{0}}$. Moreover, there exists $T \in \mathcal{L}(\mathbb{H}_{\alpha};\mathcal{Z})$ such that $(v^{0}(\cdot,v_{0}), \xi^{0}(\cdot,v_{0})) = T v_{0}$ for all $v_{0} \in \mathbb{H}_{\alpha}$ and for some $P \in \mathcal{L}(\mathbb{H}_{\alpha})$, which is self adjoint in $\mathbb{H}_{\alpha}$, we have $\mathcal{J}_{\mathcal{F}}(Tv_{0}) = (Pv,v)_{\alpha}$.
\end{lemma}
\begin{proof}
The lemma can proved in the same way as corresponding lemmas in \cite{Likhtarnikov1976}. Note that there exists $Q \in \mathcal{L}(\mathcal{Z})$, which is self-adjoint in $\mathcal{Z}$ and such that $\mathcal{J}_{\mathcal{Z}}(z) = (Qz,z)_{\mathcal{Z}}$ for every $z \in \mathcal{Z}$. Now let $v_{0} \in \mathbb{H}_{\alpha}$, $z \in \mathfrak{M}_{v_{0}}$ and $h \in \mathfrak{M}_{0}$.  We have
\begin{equation}
\mathcal{J}_{\mathcal{F}}(z+h) = (Q(z+h),z+h)_{\mathcal{Z}} = (Qz,z)_{\mathcal{Z}} + 2\operatorname{Re}(Qz,h)_{\mathcal{Z}} + (Qh,h)_{\mathcal{Z}}.
\end{equation}
For $z$ to be a minimum point of $\mathcal{J}_{\mathcal{F}}$ on $\mathfrak{M}_{v_{0}}$ it is necessary and sufficient that $(Qz,h)_{\mathcal{Z}}=0$ and $(Qh,h)_{\mathcal{Z}} \geq 0$ for all $h \in \mathfrak{M}_{0}$. Note that the latter condition implies that $\alpha_{1} \geq 0$. From Lemma \ref{LEM: OperatorSubspaceLemma} for every $z \in \mathcal{M}_{v_{0}}$ we have $z = z_{0} + Dv_{0}$, where $z_{0} \in \mathcal{M}_{0}$. Let $\Pi \colon \mathcal{Z} \to \mathfrak{M}_{0}$ be the orgthogonal projector onto $\mathfrak{M}_{0}$. Then $(Qz,h)=0$ is equivalent to $(\Pi Qz_{0} + \Pi Dv_{0},h) = 0$ or $\Pi Qz_{0} = -\Pi Dv_{0}$. Since $(\Pi Qh,h)_{\mathcal{Z}} = (Qh,h)_{\mathcal{Z}}$ for $h \in \mathfrak{M}_{0}$, from $\alpha_{1}>0$ we have that $\Pi Q \colon \mathfrak{M}_{0} \to \mathfrak{M}_{0}$ has coercive quadratic form. Therefore, by the Lax-Milgram theorem, the equation $\Pi Qz_{0} = -\Pi Dv_{0}$ has a unique solution $z_{0} = -(\Pi Q)^{-1}\Pi D v_{0}$. Then $z= z_{0} + Dv_{0} = -(\Pi Q)^{-1}\Pi D v_{0}+Dv_{0} = Tv_{0}$ is the unique optimal process and $T \in \mathcal{L}(\mathbb{H}_{\alpha};\mathcal{Z})$. Now
\begin{equation}
\mathcal{J}_{\mathcal{F}}(Tv_{0}) = (QTv_{0},Tv_{0})_{\mathcal{Z}} = (T^{*}QTv_{0},v_{0})_{\alpha}.
\end{equation}
Putting $P:=T^{*}QT$ we finish the proof.
\end{proof}

The following lemma is the main reason why the Frequency Theorem is useful in applications. For basic facts about the Fourier transform in $L_{2}$ we refer to \cite{EngelNagel2000}.

\begin{lemma}
	\label{EQ: ParabAlphaEquivalenceLemma}
	1). Suppose that $\alpha_{3}$ from \eqref{EQ: Alpha3} is well-defined, i.~e. the operator $A$ does not have spectrum on the imaginary axis and $(A-i\omega)^{-1}B$ are bounded in the norm of $\mathcal{L}(\mathbb{H}_{\alpha})$ uniformly in $\omega \in \mathbb{R}$. Then $\alpha_{3}>0$ if and only if $\alpha_{2} > 0$.
	
	2). $\alpha_{2}>0$ implies $\alpha_{1}>0$.
\end{lemma}
\begin{proof}
	1). It is clear that $\alpha_{2} \leq \alpha_{3}$. Thus $\alpha_{2}>0$ implies $\alpha_{3} > 0$. Now let us consider the value 
	\begin{equation}
	\beta := 1 + \sup_{\omega \in \mathbb{R}} \|(A-i\omega)^{-1}B\|_{\Xi \to \mathbb{H}_{\alpha}}.
	\end{equation} 
	It is easy to see that $\alpha_{2} \geq \beta^{-1} \alpha_{3}$ and, consequently, 1) is proved.
	
	2). For any process $z(\cdot)=(v(\cdot),\xi(\cdot)) \in \mathfrak{M}_{0}$ we extend it by zero to $(-\infty,0]$ and consider its Fourier transform denoted by $\hat{z}(\cdot) = (\hat{v}(\cdot), \hat{\xi}(\cdot) )$. Note that since the embedding $\mathbb{H}_{\alpha} \subset \mathbb{H}$ is continuous, the Fourier transform in the space $L_{2}(\mathbb{R};\mathbb{H} \times \Xi)$ coincides on $L_{2}(\mathbb{R};\mathbb{H}_{\alpha} \times \Xi)$ with the Fourier transform in $L_{2}(\mathbb{R};\mathbb{H}_{\alpha} \times \Xi)$. By Lemma 11 from \cite{LouisWexler1991} (or see \cite{Likhtarnikov1977}) we have $\hat{v}(\omega) \in \mathcal{D}(A)$ for almost all $\omega \in \mathbb{R}$ and
	\begin{equation}
	\label{EQ: FourierTransformProcess}
	i\omega \hat{v}(\omega) = A \hat{v}(\omega) + B\hat{\xi}(\omega).
	\end{equation}
	By the Plancherel theorem for the Fourier transform \cite{EngelNagel2000} we have
	\begin{equation}
	\begin{split}
	\int_{0}^{+\infty} \mathcal{F}(v(t),\xi(t))dt = \int_{-\infty}^{+\infty} \mathcal{F}(\hat{v}(\omega),\hat{\xi}(\omega))d\omega \geq \\ \geq \alpha_{2} \int_{-\infty}^{+\infty}\left(|\hat{v}(\omega)|^{2}_{\alpha}+|\hat{\xi}(\omega)|^{2}_{\Xi} \right) d\omega = \alpha_{2}\int_{0}^{\infty} \left(|v(t)|^{2}_{\alpha} + |\xi(t)|^{2}_{\Xi} \right)dt.
	\end{split}
	\end{equation}
	Thus, $\alpha_{1} \geq \alpha_{2}$ and the proof is finished.
\end{proof}

It is remained to show that $\alpha_{1}>0$ implies that $\alpha_{2} > 0$. This will be done in the proof of Theorem \ref{TH: OptimalProcessParab} below with the aid of the following lemma (see Lemma 5 in \cite{Anikushin2020FreqDelay} for a proof).
\begin{lemma}
	\label{LEM: ConstantDifferentiationProperty}
	Let $f(\cdot)$ be a twice continuously differentiable $\mathbb{H}$-valued function. Suppose that $v_{0} \in \mathcal{D}(A)$ is such that $Av_{0} + f(0) \in \mathcal{D}(A)$. Then for the classical solution $v(\cdot)$ of
	\begin{equation}
	\label{EQ: InhomogeneousConstantEquation}
	\dot{v}(t) = A v(t) + f(t)
	\end{equation}
	with $v(0)=v_{0}$ we have
	\begin{equation}
	\lim\limits_{h \to 0+}\frac{v(h)-v(0)}{h} = v'(0)=Av(0)+f(0),
	\end{equation}
	where the limit exists in $\mathcal{D}(A)$ endowed with the graph norm.
\end{lemma}

The following theorem generalizes corresponding results from \cite{Likhtarnikov1977, LouisWexler1991}, which become a special case of our result if $\mathbb{H}_{\alpha}=\mathbb{H}$.

\begin{theorem}
	\label{TH: OptimalProcessParab}
	Let the pair $(A,B)$ satisfy \textbf{(CONT)}. We have the following:
	
	1) If $\alpha_{1}>0$, then for every $v_{0} \in \mathbb{H}_{\alpha}$ the quadratic functional \eqref{EQ: QuadraticFunctional} has a unique minimum $(v^{0}(\cdot,v_{0}),\xi^{0}(\cdot,v_{0}))$ on $\mathfrak{M}_{v_{0}}$ and the map
	\begin{equation}
	\mathbb{H}_{\alpha} \ni v_{0} \mapsto (v^{0}(\cdot,v_{0}),\xi^{0}(\cdot,v_{0})) \in \mathcal{Z}
	\end{equation}
	is a linear bounded operator. Moreover, there exists $P \in \mathcal{L}(\mathbb{H}_{\alpha})$, which is self-adjoint in $\mathbb{H}_{\alpha}$ and such that for all $v_{0} \in \mathbb{H}_{\alpha}$ we have
	\begin{equation}
	(Pv_{0}, v_{0} )_{\alpha} = \mathcal{J}_{\mathcal{F}}((v^{0}(\cdot,v_{0}),\xi^{0}(\cdot,v_{0})).
	\end{equation}
	For the quadratic functional $V(v):=( Pv, v )_{\alpha}$ and any $T \geq 0$ we have
	\begin{equation}
	\label{EQ: OperatorSolutionsInequality}
	V(v(T)) - V(v_{0}) + \int_{0}^{T} \mathcal{F}(v(t),\xi(t))dt \geq 0,
	\end{equation}
	where $v(t)=v(t,v_{0},\xi)$ is the solution to \eqref{EQ: ControlSystem} with arbitrary $v(0)=v_{0} \in \mathbb{H}_{\alpha}$ and $\xi(\cdot) \in L_{2}(0,T;\Xi)$ such that $v(\cdot) \in C([0,T];\mathbb{H}_{\alpha})$.
	
	2) $\alpha_{1}>0$ is equivalent to $\alpha_{2}>0$. If $\alpha_{3}$ is well-defined in the above given sense, then $\alpha_{2}>0$ is equivalent to $\alpha_{3}>0$.
\end{theorem}
\begin{proof}
	The first part of item 1) follows from Lemma \ref{LEM: ParabOptimizationOnAffineSubspaces}. Let $T \geq 0$ and $v_{0} \in \mathbb{H}_{\alpha}$ be fixed and suppose we are given with a solution $v(t)=v(t,v_{0},\xi)$, where $t \in [0,T]$, to \eqref{EQ: ControlSystem} with $v(0)=v_{0}$ and $\xi(\cdot) \in L_{2}(0,T;\Xi)$ such that $v(\cdot) \in C([0,T];\mathbb{H}_{\alpha})$. Consider the process $(\widetilde{v}(\cdot),\widetilde{\xi}(\cdot)) \in \mathfrak{M}_{v_{0}}$, where
	\begin{equation}
	\widetilde{v}(t)= \begin{cases}
	v(t), &\text{ if } t \in [0,T],\\
	v^{0}(t-T,v(T)), &\text{ if } t > T.
	\end{cases}
	\end{equation}
	and
	\begin{equation}
	\widetilde{\xi}(t)= \begin{cases}
	\xi(t), &\text{ if } t \in [0,T],\\
	\xi^{0}(t-T,v(T)), &\text{ if } t > T.
	\end{cases}
	\end{equation}
	From the inequality
	\begin{equation}
	V(v_{0})=\mathcal{J}_{\mathcal{F}}(v^{0}(\cdot,v_{0}),\xi^{0}(\cdot,v_{0})) \leq \mathcal{J}_{\mathcal{F}}(\widetilde{v}(\cdot),\xi(\cdot))=V(v(T)) + \int_{0}^{T}\mathcal{F}(v(t),\xi(t))
	\end{equation}
	we have \eqref{EQ: OperatorSolutionsInequality} satisfied. Thus, the second part of item 1) is proved. Due to Lemma \ref{EQ: ParabAlphaEquivalenceLemma} to prove item 2) it is enough to show that $\alpha_{1}>0$ implies $\alpha_{2}>0$. Let us suppose that $\alpha_{1}>0$. Consider the form $\mathcal{F}_{\delta}(v,\xi):=\mathcal{F}(v,\xi)-\delta (|v|^{2}_{\alpha}+|\xi|_{\Xi}^{2})$. It is clear that for the form $\mathcal{F}_{\delta}$ the condition analogous to $\alpha_{1}>0$ will be satisfied if $\delta>0$ is chosen sufficiently small. Therefore, there exists an operator $P_{\delta} \in \mathcal{L}(\mathbb{H}_{\alpha})$, which is self-adjoint and such that for $h>0$ we have
	\begin{equation}
	\label{EQ: MonotoneIneqAlpha1ImpliesA2}
	( P_{\delta} v(h),v(h) )_{\alpha} - (  P_{\delta}v_{0},v_{0} ) + \int_{0}^{h} \mathcal{F}(v(s),\xi(s)) \geq \delta \int_{0}^{h}(|v(s)|^{2}_{\alpha}+|\xi(s)|_{\Xi}^{2})ds,
	\end{equation}
	where $v(\cdot)=v(\cdot,v_{0},\xi)$ and $\xi(\cdot) \in L_{2}(0,h;\Xi)$ such that $v(\cdot) \in C([0,h];\mathbb{H}_{\alpha})$. Let us choose $\xi(\cdot) \equiv \xi_{0}$ for a fixed $\xi_{0} \in \Xi$ and $v_{0} \in \mathcal{D}(A)$ such that $Av_{0} + B\xi_{0} \in \mathcal{D}(A)$. Then $v(\cdot) \in C([0,h];\mathcal{D}(A))$. From Lemma \ref{LEM: ConstantDifferentiationProperty} it follows that $v(\cdot)$ is $\mathcal{D}(A)$-differentiable at $t=0$ and since the embedding $\mathcal{D}(A) \subset \mathbb{H}_{\alpha}$ is continuous, it is also $\mathbb{H}_{\alpha}$-differentiable. Dividing \eqref{EQ: MonotoneIneqAlpha1ImpliesA2} by $h>0$ and letting $h$ tend to zero, we get
	\begin{equation}
	2\operatorname{Re}( Av_{0}+B\xi_{0}, P_{\delta}v_{0} )_{\alpha} + \mathcal{F}(v_{0},\xi_{0}) \geq \delta ( |v_{0}|^{2}_{\alpha}+|\xi_{0}|_{\Xi}^{2} ).
	\end{equation}
	If we choose $v_{0}$ and $\xi_{0}$ such that $i \omega v_{0} = A v_{0} + B \xi_{0}$ for some $\omega \in \mathbb{R}$, we immediately get
	\begin{equation}
	\mathcal{F}(v_{0},\xi_{0}) \geq \delta ( |v_{0}|^{2}_{\alpha}+|\xi_{0}|_{\Xi}^{2} ),
	\end{equation}
	which implies $\alpha_{2}>0$. Thus, the proof is finished.
\end{proof}
\section{Relaxing the controllability}
\label{SEC: RelaxingControllability}

Let us consider a modification of \eqref{EQ: ControlSystem} given by
\begin{equation}
\label{EQ: ModifiedControlSystem}
\dot{v}(t) = Av(t) + B\xi(t) + \eta(t) = Av(t) + \widetilde{B}\zeta(t),
\end{equation}
where $\zeta(t)=(\xi(t),\eta(t))$ is a new control (and $\widetilde{B}\zeta(t) = B\xi(t) + \zeta(t)$) with values from the extended control space $\Xi \times \mathbb{H}$.

Suppose that $A$ is also the generator of a $C_{0}$-semigroup in $\mathbb{H}_{\alpha}$. Then it is clear that the pair $(A,\widetilde{B})$ satisfies \textbf{(CONT)} no matter what $B$ is. Indeed, there exists $\nu_{0} \in \mathbb{R}$ and $M>0$ such that 
\begin{equation}
\|G(t)\|_{\mathbb{H}_{\alpha} \to \mathbb{H}_{\alpha}} \leq M e^{\nu_{0} t}.
\end{equation}
The only interesting case is $\nu_{0} \geq 0$. Let us fix any $\nu > \nu_{0}$. Now for $v_{0} \in \mathbb{H}_{\alpha}$ consider the mild solution $w(\cdot)$ such that $w(0)=v_{0}$ and
\begin{equation}
\dot{w}(t) = (A - \nu I)w(t).
\end{equation} 
Clearly, $w(\cdot) \in L_{2}(0,+\infty;\mathbb{H}_{\alpha})$ since $w(t)=e^{-\nu t}G(t)v_{0}$. Putting $\xi(\cdot) \equiv 0$ and $\eta(\cdot):= -\nu w(\cdot)$, we get that $w(\cdot)$ satisfies \eqref{EQ: ModifiedControlSystem} and, consequently, \textbf{(CONT)} holds for the pair $(A,\widetilde{B})$.

For $\gamma > 0$ we consider the quadratic form
\begin{equation}
\mathcal{F}_{\gamma}(v,\xi,\eta) := \mathcal{F}(v,\xi) + \gamma |\eta|^{2}_{\mathbb{H}}.
\end{equation}
We also put
\begin{equation}
\alpha_{\gamma} := \inf \frac{ \mathcal{F}_{\gamma}(v,\xi,\eta) }{|v|^{2}_{\mathbb{H}}+|\xi|^{2}_{\Xi} + |\eta|^{2}_{\mathbb{H}}},
\end{equation}
where the infimum is taken over all $\omega \in \mathbb{R}$, $v \in \mathcal{D}(A)$, $\xi \in \Xi$ and $\eta \in \mathbb{H}$ such that $i \omega v = A v + B \xi + \eta$. So, the inequality $\alpha_{\gamma}>0$ allows us to apply Theorem \ref{TH: OptimalProcessParab} to the pair $(A,\widetilde{B})$ and the form $\mathcal{F}_{\gamma}$. Now we can prove the following theorem (and Theorem \ref{TH: ParabComplexFreqTheorem} in particular), which relaxes the controllability assumption.
\begin{theorem}
	\label{TH: FreqThRelaxedDelay}
	Let $A$ be also the generator of a $C_{0}$-semigroup in $\mathbb{H}_{\alpha}$ and let \textbf{(RES)} be satisfied. Then the following conditions for the pair $(A,B)$ and the form $\mathcal{F}$ are equivalent:
	
	1. $\alpha_{2}>0$
	
	2. There exists $P \in \mathcal{L}(\mathbb{H}_{\alpha})$, self-adjoint in $\mathbb{H}_{\alpha}$ and such that for $V(v):=( Pv, v )_{\alpha}$ and any $T \geq 0$ we have
	\begin{equation}
	\label{EQ: FreqThParabRelaxingOperatorIneq}
	V(v(T)) - V(v_{0}) + \int_{0}^{T} \mathcal{F}(v(t),\xi(t))dt \geq 0,
	\end{equation}
	where $v(t)=v(t,v_{0},\xi)$ is the solution to \eqref{EQ: ControlSystem} with arbitrary $v(0)=v_{0} \in \mathbb{H}_{\alpha}$ and $\xi(\cdot) \in L_{2}(0,T;\Xi)$ such that $v(\cdot) \in C([0,T];\mathbb{H}_{\alpha})$.
	
	Moreover, $\alpha_{2}>0$ is equivalent to $\alpha_{3}>0$.
\end{theorem}

\begin{proof}
	Our aim is to show that $\alpha_{2}>0$ (given by \eqref{EQ: Alpha2} for the form $\mathcal{F}$) implies that $\alpha_{\gamma} > 0$ for a sufficiently large $\gamma>0$ and therefore we can apply Theorem \ref{TH: OptimalProcessParab} to the pair $(A,\widetilde{B})$ and the form $\mathcal{F}_{\gamma}$ to get an operator $P_{\gamma} \in \mathcal{L}(\mathbb{H}_{\alpha})$ such that for $V(v):= (P_{\gamma}v, v )_{\alpha}$ the inequality 
	\begin{equation}
	V(v(T)) - V(v_{0}) + \int_{0}^{T} \mathcal{F}_{\gamma}(v(t),\xi(t),\eta(t))dt \geq 0
	\end{equation}
	is satisfied for any solution $v(\cdot)=v(\cdot,v_{0},(\xi(\cdot),\eta(\cdot)))$ of \eqref{EQ: ModifiedControlSystem}, where $\xi(\cdot) \in L_{2}(0,T;\Xi)$, $\eta(\cdot) \in L_{2}(0,T;\mathbb{H})$ and such that $v(\cdot) \in C([0,T];\mathbb{H}_{\alpha})$. Putting $\eta(\cdot) \equiv 0$, we get inequality \eqref{EQ: FreqThParabRelaxingOperatorIneq} for the form $\mathcal{F}$ and the pair $(A,B)$. Thus $P:=P_{\gamma}$ satisfies the required properties.
	
	The necessity of $\alpha_{2}>0$ for the existence of $P$ was shown in the proof of Theorem \ref{TH: OptimalProcessParab}. Let us suppose that $\alpha_{2}>0$. Then there exists $\varepsilon>0$ such that the inequality
	\begin{equation}
	\label{EQ: RelaxationStrongerIneq}
	\mathcal{F}(v,\xi) \geq \varepsilon ( |v|^{2}_{\alpha} + |\xi|^{2}_{\Xi} )
	\end{equation}
	is satisfied for all $v = (i\omega I - A)^{-1}B\xi$, where $\omega \in \mathbb{R}$ and $\xi \in \Xi$ are arbitrary. Let us for every $\eta \in \mathbb{H}$ and $\omega \in \mathbb{R}$ consider vectors $v_{\omega}(\eta) \in \mathcal{D}(A)$ and $\xi_{\omega}(\eta) \in \Xi$ such that we have
	\begin{equation}
	i \omega v_{\omega}(\eta) = A v_{\omega}(\eta) + B \xi_{\omega}(\eta) + \eta.
	\end{equation}
	and there are constants $M_{1}>0, M_{2}>0$ such that
	\begin{equation}
	\label{EQ: RelaxationFunctionsConstants}
	\| v_{\omega}(\eta) \|_{\alpha} \leq M_{1} |\eta|_{\mathbb{H}} \text{ and } |\xi_{\omega}(\eta)|_{\Xi} \leq M_{2} |\eta|_{\mathbb{H}}.
	\end{equation}
	Since \textbf{(RES)} is satisfied, we can take, for example, $v_{\omega}(\eta):= (i\omega I - A)^{-1}\eta$ and $\xi_{\omega}(\eta) = 0$. Now suppose $\omega \in \mathbb{R}$, $v \in \mathcal{D}(A)$, $\xi \in \Xi$ and $\eta \in \mathbb{H}$ are given such that
	\begin{equation}
	i \omega v = A v + B \xi + \eta.
	\end{equation}
	Put $\delta v := v - v_{\omega}(\eta)$, $\delta \xi := \xi - \xi_{\omega}(\omega)$. Then we have
	\begin{equation}
	\label{EQ: RelaxationDeltaFreq}
	i\omega \delta v = A \delta v + B \delta \xi
	\end{equation}
	and, consequently,
	\begin{equation}
	\mathcal{F}_{\gamma}(v,\xi,\eta) = \mathcal{F}(v_{\omega}(\eta),\xi_{\omega}(\eta)) + L(\delta v,\delta \xi; v_{\omega}(\eta),\xi_{\omega}(\eta)) + \mathcal{F}(\delta v,\delta \xi) + \gamma |\eta|^{2}_{\mathbb{H}}.
	\end{equation}
	From \eqref{EQ: RelaxationDeltaFreq} and \eqref{EQ: RelaxationStrongerIneq} we have that $\mathcal{F}(\delta v,\delta \xi) \geq \varepsilon ( |\delta v|^{2}_{\alpha} + |\delta\xi|^{2}_{\Xi} )$. Due to \eqref{EQ: RelaxationFunctionsConstants} there are constants $\widetilde{M}_{1}>0$, $\widetilde{M}_{2}>0$ such that
	\begin{equation}
	\label{EQ: RelaxingFormConstants}
	|\mathcal{F}(v_{\omega}(\eta),\xi_{\omega}(\eta))| \leq \widetilde{M}_{1}|\eta|^{2}_{\mathbb{H}} \text{ and } |L(\delta v, \delta \xi; v_{\omega}(\eta),\xi_{\omega}(\eta))| \leq \widetilde{M}_{2} |\delta(\xi)|_{\Xi} \cdot |\eta|_{\mathbb{H}}.
	\end{equation}
	Let $\gamma_{1} := \widetilde{M}_{1}$ and $\gamma_{2} > 0$ be such that $\gamma_{2} x^{2} - \widetilde{M}_{2} \alpha x + \alpha^2 \varepsilon/2 \geq 0$ for all $x \in \mathbb{R}$ and all $\alpha \geq 0$. The latter will be satisfied if $2 \gamma_{2} \varepsilon \geq \widetilde{M}^{2}_{2}$. Put $\gamma := \gamma_{1} + \gamma_{2} + \gamma_{3}$ for any $\gamma_{3}>0$. From this and \eqref{EQ: RelaxingFormConstants} we have
	\begin{equation}
	\mathcal{F}_{\gamma}(v,\xi,\eta) \geq \frac{1}{2}\varepsilon ( |\delta v|^{2}_{\alpha} + |\delta\xi|^{2}_{\Xi} ) + \gamma_{3} |\eta|^{2}_{\mathbb{H}}.
	\end{equation}
	Since $|v|_{\alpha} \leq M_{1}|\eta|_{\mathbb{H}} + |\delta v|_{\alpha}$ and $|\xi|_{\Xi} \leq M_{2} |\eta|_{\mathbb{H}} + |\delta \xi|_{\Xi}$, this implies that $\alpha_{\gamma} > 0$. The proof is finished.
\end{proof}
\section{Realification of the operator $P$ and its sign properties}
\label{SEC: Realification}
In applications, we usually encounter control systems posed in the context of real spaces $\mathbb{H}$, $\mathbb{H}_{\alpha}$ and $\Xi$ with the real operators $A$ and $B$. Suppose a real quadratic form $\mathcal{F}(v,\xi)$ of $v \in \mathbb{H}_{\alpha}$ and $\xi \in \Xi$ is given.  Let us consider its Hermitian extension to the complexifications $\mathbb{H}^{\mathbb{C}}$, $\mathbb{H}^{\mathbb{C}}_{\alpha}$ and $\Xi^{\mathbb{C}}$ given by $\mathcal{F}^{\mathbb{C}}(v_{1}+iv_{2}, \xi_{1}+i\xi_{2}):=\mathcal{F}(v_{1},\xi_{2}) + \mathcal{F}(v_{2},\xi_{2})$ for $v_{1},v_{2} \in \mathbb{H}_{\alpha}$ and $\xi_{1},\xi_{2} \in \Xi$. Suppose $\mathcal{F}$ has the form
\begin{equation}
	\label{EQ: RealQuadraticFormParab}
	\mathcal{F}(v,\xi) = (\mathcal{F}_{1}v,v)_{\mathbb{H}_{\alpha}} + 2(\mathcal{F}_{2}v,\xi)_{\Xi} + (\mathcal{F}_{3}\xi,\xi)_{\Xi} \text{ for } v \in \mathbb{H}_{\alpha} \text{ and } \xi \in \Xi,
\end{equation}
where $\mathcal{F}^{*}_{1}=\mathcal{F}_{1} \in \mathcal{L}(\mathbb{H}_{\alpha})$, $\mathcal{F}_{2} \in \mathcal{L}(\mathbb{H}_{\alpha};\Xi)$ and $\mathcal{F}^{*}_{3}=\mathcal{F}_{3} \in \mathcal{L}(\Xi)$. Then it is easy to check that $\mathcal{F}^{\mathbb{C}}$ is given by
\begin{equation}
	\mathcal{F}^{\mathbb{C}}(v,\xi)=(\mathcal{F}_{1}^{\mathbb{C}}v,v)_{\mathbb{H}^{\mathbb{C}}_{\alpha}} + 2\operatorname{Re}(\mathcal{F}^{\mathbb{C}}_{2}v,\xi)_{\Xi^{\mathbb{C}}} + (\mathcal{F}^{\mathbb{C}}_{3}\xi,\xi)_{ \Xi^{\mathbb{C}} } \text{ for } v \in \mathbb{H}^{\mathbb{C}}_{\alpha} \text{ and } \xi \in \Xi^{\mathbb{C}},
\end{equation}
where $\mathcal{F}^{\mathbb{C}}_{1}, \mathcal{F}^{\mathbb{C}}_{2}, \mathcal{F}^{\mathbb{C}}_{3}$ are complexifications of the operators $\mathcal{F}_{1},\mathcal{F}_{2},\mathcal{F}_{3}$ respectively.

For the real context we have the following version of Theorem \ref{TH: ParabComplexFreqTheorem}.
\begin{theorem}
	\label{TH: ParabRealFreqTheorem}
	Suppose $A$ is the generator of a $C_{0}$ semigroup in $\mathbb{H}_{\alpha}$ also and let \textbf{(RES)} hold. The following conditions for the pair $(A,B)$ and the form $\mathcal{F}$ are equivalent:
	
	1. For some $\delta'>0$ we have $\mathcal{F}^{\mathbb{C}}(-(A^{\mathbb{C}} - i\omega I)B^{\mathbb{C}}\xi,\xi) \leq -\delta' |\xi|^{2}_{\Xi^{\mathbb{C}}}$ for all $\xi \in \Xi^{\mathbb{C}}$ and $\omega \in \mathbb{R}$.
	
	2. There are $\delta>0$ and $P \in \mathcal{L}(\mathbb{H}_{\alpha})$, which is self-adjoint in $\mathbb{H}_{\alpha}$ and such that for $V(v):=( Pv, v )_{\alpha}$ and any $T \geq 0$ we have
	\begin{equation}
		\label{EQ: RealFreqThInequality}
		V(v(T)) - V(v_{0}) + \int_{0}^{T} \mathcal{F}(v(s),\xi(s))ds \leq -\delta \int_{0}^{T}(|v(s)|^{2}_{\mathbb{H}_\alpha}+|\xi(s)|^{2}_{\Xi})ds,
	\end{equation}
	where $v(t)=v(t,v_{0},\xi)$ is the solution to \eqref{EQ: ControlSystem} with arbitrary $v(0)=v_{0} \in \mathbb{H}_{\alpha}$ and $\xi(\cdot) \in L_{2}(0,T;\Xi)$ such that $v(\cdot) \in C([0,T];\mathbb{H}_{\alpha})$.
\end{theorem}
\begin{proof}
	Let us show that item $1$ implies item $2$. For this we apply Theorem \ref{TH: ParabComplexFreqTheorem} to the operators $A^{\mathbb{C}}$ and $B^{\mathbb{C}}$, spaces $\mathbb{H}^{\mathbb{C}}$, $\Xi^{\mathbb{C}}$, $\mathbb{H}_{\alpha}^{\mathbb{C}}$ and the form $\mathcal{F}_{\delta}(v,\xi):=-\mathcal{F}^{\mathbb{C}}(v,\xi) + \delta (|v|^{2}_{\mathbb{H}^{\mathbb{C}}_{\alpha}} + |\xi|^{2}_{\Xi^{\mathbb{C}}})$ for a sufficiently small $\delta > 0$ such that $\alpha_{2}(\mathcal{F}_{\delta},A^{\mathbb{C}},B^{\mathbb{C}}) > 0$. Thus, there exists $-\widetilde{P} \in \mathcal{L}(\mathbb{H}^{\mathbb{C}}_{\alpha})$, self-adjoint in $\mathbb{H}^{\mathbb{C}}_{\alpha}$ and such that
	\begin{equation}
	\label{EQ: ComplexInequality}
	\begin{split}
	\left( \widetilde{P}(v_{1}(T)+ iv_{2}(T)), v_{1}(T)+ iv_{2}(T) \right)_{\mathbb{H}_{\alpha}} - \left(\widetilde{P}(v_{0,1}+ iv_{0,2}), v_{0,1}+ iv_{0,2} \right)_{\mathbb{H}_{\alpha}} +\\+ \int_{0}^{T} \mathcal{F}(v_{1}(t),\xi_{1}(t))dt + \int_{0}^{T}\mathcal{F}(v_{2}(t),\xi_{2}(t))dt \leq -\delta \int_{0}^{T}(|v(s)|^{2}_{\mathbb{H}^{\mathbb{C}}_{\alpha}}+|\xi(s)|^{2}_{\Xi^{\mathbb{C}}})ds,
	\end{split}	
	\end{equation}
	where $v_{1}(t)=v(t,v_{0,1},\xi_{1})$ and $v_{2}(t)=v(t,v_{0,2},\xi_{2})$ are arbitrary solutions in the real spaces. Note that $\widetilde{P}$ can be represented as
	\begin{equation}
	\widetilde{P} = \begin{bmatrix}
	\widetilde{P}_{11} \ \widetilde{P}_{12}\\
	\widetilde{P}_{21} \ \widetilde{P}_{22}
	\end{bmatrix},
	\end{equation}
	where $\widetilde{P}_{ij} \in \mathcal{L}(\mathbb{H}_{\alpha})$ for $i,j \in \{ 1,2 \}$ and, moreover, $( v, \widetilde{P}_{12}v )_{\mathbb{H}_{\alpha}} = ( v, \widetilde{P}_{21}v )_{\mathbb{H}_{\alpha}} = 0$ for all $v \in \mathbb{H}_{\alpha}$. Putting $v_{2}(\cdot) \equiv 0$ and $\xi_{2}(\cdot) \equiv 0$ in \eqref{EQ: ComplexInequality}, we get \eqref{EQ: RealFreqThInequality} for $P := \widetilde{P}_{11}$.
	
	The inverse implication can be shown analogously to the corresponding part of the proof in Theorem \ref{TH: OptimalProcessParab}, where Lemma \ref{LEM: ConstantDifferentiationProperty} is used.
\end{proof}

We say that a self-adjoint operator $P \in \mathcal{L}(\mathbb{H}_{\alpha})$ is \textit{positive} (resp. \textit{negative}) on $\mathbb{L} \subset \mathbb{H}_{\alpha}$ if $(Pv,v)_{\mathbb{H}_{\alpha}} > 0$ (resp. $(Pv,v)_{\mathbb{H}_{\alpha}} < 0$) for all non-zero $v \in \mathbb{L}$.

Now suppose that there exists a decomposition of $\mathbb{H}_{\alpha}$ into the direct sum of two subspaces $\mathbb{H}^{s}_{\alpha}$ and $\mathbb{H}^{u}_{\alpha}$, i.~e. $\mathbb{H}_{\alpha} = \mathbb{H}^{s}_{\alpha} \oplus \mathbb{H}^{u}_{\alpha}$, such that for $v_{0} \in \mathbb{H}^{s}_{\alpha}$ we have $G(t)v_{0} \to 0$ in $\mathbb{H}_{\alpha}$ as $t \to +\infty$ and any $v_{0} \in \mathbb{H}^{u}_{\alpha}$ admits a backward extension\footnote{That is a continuous function $v \colon \mathbb{R} \to \mathbb{H}_{\alpha}$ such that $v(0)=v_{0}$ and $v(t+s)=G(t)v(s)$ for all $t \geq 0$ and $s \in \mathbb{R}$.} $v(\cdot)$ such that $v(t) \to 0$ in $\mathbb{H}_{\alpha}$ as $t \to -\infty$.

\begin{theorem}
	\label{TH: HighRankConesTheorem}
	Assume that there is a decomposition $\mathbb{H}_{\alpha}=\mathbb{H}^{s}_{\alpha} \oplus \mathbb{H}^{u}_{\alpha}$ as above and $\operatorname{dim}\mathbb{H}^{u}_{\alpha} =: j < \infty$. Suppose that for an operator $P \in \mathcal{L}(\mathbb{H}_{\alpha})$, self-adjoint in $\mathbb{H}_{\alpha}$, and the quadratic form $V(v):=(Pv,v)_{\mathbb{H}_{\alpha}}$ we have the Lyapunov inequality 
	\begin{equation}
	\label{EQ: LyapunovInequalityIntegralForm}
	V(v(T)) - V(v(0)) \leq -\delta \int_{0}^{T}|v(t)|^{2}_{\mathbb{H}_{\alpha}}dt
	\end{equation}
	satisfied for all $T>0$, $v_{0} \in \mathbb{H}_{\alpha}$ and $v(t)=G(t)v_{0}$.
	
	Then $P$ is positive on $\mathbb{H}^{s}_{\alpha}$ and negative on $\mathbb{H}^{u}_{\alpha}$. Moreover, the set $\mathcal{C}_{V}:=\{ v \in \mathbb{H}_{\alpha} \ | \ V(v) \leq 0 \}$ is a $j$-dimensional cone in $\mathbb{H}_{\alpha}$ in the sense that
	\begin{enumerate}
		\item[1)] $\mathcal{C}_{V}$ is closed in $\mathbb{H}_{\alpha}$;
		\item[2)] $\alpha v \in \mathcal{C}_{V}$ for all $v \in \mathcal{C}_{V}$ and $\alpha \in \mathbb{R}$;
		\item[3)] We have
		\begin{equation}
		j=\max_{\mathbb{F}}\operatorname{dim}\mathbb{F} =: \operatorname{d}(\mathcal{C}_{V}),
		\end{equation}
		where the maximum is taken over all linear subspaces $\mathbb{F} \subset \mathbb{H}_{\alpha}$ such that $\mathbb{F} \subset \mathcal{C}_{V}$.
	\end{enumerate}
\end{theorem}
\begin{proof}
	Taking it to the limit as $T \to +\infty$ in \eqref{EQ: LyapunovInequalityIntegralForm} for $v_{0} \in \mathbb{H}^{s}_{\alpha}$ we have
	\begin{equation}
	\label{EQ: PositivenessLyapunovInt}
	V(v_{0}) \geq \delta \int_{0}^{+\infty}|v(t)|^{2}_{\mathbb{H}_{\alpha}}dt.
	\end{equation}
	Analogously, for $v_{0} \in \mathbb{H}^{u}_{\alpha}$ we have
	\begin{equation}
	V(v_{0}) \leq -\delta\int_{-\infty}^{0}|v(t)|^{2}_{\mathbb{H}_{\alpha}}dt.
	\end{equation}
	Therefore, $V(v) > 0$ for all $v \in \mathbb{H}^{s}_{\alpha}$, $v \not=0$, and $V(v) < 0$ for all $v \in \mathbb{H}^{u}_{\alpha}$, $v \not= 0$. So, the sign properties of $P$ are established.
	
	Properties 1), 2) are obvious. Let a subspace $\mathbb{F} \subset \mathbb{H}_{\alpha}$ such that $\mathbb{F} \subset \mathcal{C}_{V}$ be given. We fix $k>j$ vectors $e_{1},\ldots, e_{k} \in \mathbb{H}_{\alpha}$. Since $\mathbb{H}_{\alpha} = \mathbb{H}^{s}_{\alpha} \oplus \mathbb{H}^{u}_{\alpha}$, for all $i=1,\ldots,k$ there exists a unique decomposition
	\begin{equation}
	e_{i} = e^{s}_{i} + e^{u}_{i},
	\end{equation}
	where $e^{s}_{i} \in \mathbb{H}^{s}_{\alpha}$ and $e^{u}_{i} \in \mathbb{H}^{u}_{\alpha}$. Since $k > \operatorname{dim} \mathbb{H}^{u}_{\alpha}$ there are constants $c_{i}$ such that
	\begin{equation}
	\sum_{i=1}^{k} c_{i} e^{u}_{i} = 0.
	\end{equation}
	From this we have
	\begin{equation}
	\label{EQ: ConeDecompLemma}
    \sum_{i=1}^{k} c_{i} e_{i} = \sum_{i=1}^{k} c_{i} e^{s}_{i} + \sum_{i=1}^{k} c_{i} e^{u}_{i} = \sum_{i=1}^{k} c_{i} e^{s}_{i}.
	\end{equation}
	From \eqref{EQ: ConeDecompLemma} and since $\mathbb{F} \subset \mathcal{C}_{V}$, we must have
	\begin{equation}
	0 \geq V\left(\sum_{i=1}^{k} c_{i} e_{i}\right) = V\left(\sum_{i=1}^{k} c_{i} e^{s}_{i}\right) \geq 0.
	\end{equation}
	Thus, $V\left(\sum_{i=1}^{k} c_{i} e^{s}_{i}\right) = 0$ or, in virtue of \eqref{EQ: PositivenessLyapunovInt}, $\sum_{i=1}^{k} c_{i} e^{s}_{i}=0$. But from \eqref{EQ: ConeDecompLemma} it follows that $e_{1},\ldots,e_{k}$ are linearly dependent. Since this holds for any $k > j$ vectors in $\mathbb{F}$, we have  $\operatorname{dim} \mathbb{F} \leq j$. Clearly, for $\mathbb{F} := \mathbb{H}^{u}_{\alpha}$ we have $\mathbb{F} \subset \mathcal{C}_{V}$ and $\dim \mathbb{F} = j$. Thus, $d(\mathcal{C}_{V}) = j$. The proof is finished.
\end{proof}
\section{Inertial manifolds for semilinear parabolic equations and frequency-domain methods}
\label{SEC: IMParabSemilinear}

\subsection{\textbf{Preliminaries}}
Let $A \colon \mathcal{D}(A) \subset \mathbb{H} \to \mathbb{H}$ be a self-adjoint positive-definite operator acting in a real Hilbert space $\mathbb{H}$ and such that $A^{-1} \colon \mathbb{H} \to \mathbb{H}$ is compact. Then one can define powers of $A$ and a scale of Hilbert spaces $\mathbb{H}_{\alpha}:=\mathcal{D}(A^{\alpha})$ with the scalar products $( v_{1},v_{2})_{\alpha}:=(A^{\alpha}v_{1},A^{\alpha}v_{2})$ for all $v_{1},v_{2} \in \mathcal{D}(A^{\alpha})$ (see, for example, the monograph of I.~Chueshov \cite{Chueshov2015}). Let $\alpha \in [0,1)$ be fixed and let $\mathbb{M},\Xi$ be some Hilbert spaces. Suppose that linear bounded operators $C \colon \mathbb{H}_{\alpha} \to \mathbb{M}$ and $B \colon \Xi \to \mathbb{H}$ are given. Let $W \colon \mathbb{R} \to \mathbb{H}$ be a continuous function and let $F \colon \mathbb{R} \times \mathbb{M} \to \Xi$ be a continuous map such that for some constant $\Lambda>0$ and all $y_{1},y_{2} \in \mathbb{M}$, $t \in \mathbb{R}$ we have
\begin{equation}
|F(t,y_{1})-F(t,y_{2})|_{\Xi} \leq \Lambda |y_{1}-y_{2}|_{\mathbb{M}}.
\end{equation}

Finally, let $K \in \mathcal{L}(\mathbb{H}_{\alpha};\mathbb{H})$. We consider the nonlinear evolution equation
\begin{equation}
\label{EQ: ParabolicNonAutEquation}
\dot{v}(t) = -Av(t) + Kv(t) + BF(t,Cv(t)) + W(t).
\end{equation}

\begin{remark}
	There is a more general situation, in which $A$ is a sectorial operator, treated by D.~Henry in \cite{Henry1981}. However, these results require H\"{o}lder-like properties in $t$ for the nonlinearity $F$, which we do not want to assume. However, these restrictions can be relaxed if we use the approach based on the Banach fixed point theorem from Theorem 4.2.3 in \cite{Chueshov2015}, which also works for the case when $A$ is sectorial. So, our results below remains true for the case when $A$ is sectorial.
\end{remark}
Following the fixed point arguments from Theorem 4.2.3 in \cite{Chueshov2015}, it can be shown that for any $v_{0} \in \mathbb{H}_{\alpha}$ and $t_{0} \in \mathbb{R}$ there exists a unique mild solution $v(t)=v(t,t_{0},v_{0})$ defined for $t \geq t_{0}$ and such that $v(t_{0})=v_{0}$. For such $v(\cdot)$ we have the variation of constants formula:
\begin{equation}
	v(t)=G(t-t_{0})v_{0} + \int_{t_{0}}^{t}G(t-s)(Kv(s)+BF(s,Cv(s))+W(s))ds
\end{equation}
satisfied. Moreover, for every $T>0$ there is a constant $C=C(T)>0$ such that the estimate 
\begin{equation}
|v(t;t_{0},v_{1}) - v(t;t_{0},v_{2})|_{\alpha} \leq C |v_{1}-v_{2}|_{\alpha}
\end{equation}
holds for all $v_{1},v_{2} \in \mathbb{H}_{\alpha}$, $t_{0} \in \mathbb{R}$ and $t \in [t_{0},t_{0}+T]$. From the integral relation (which defines mild solutions) it is easy to show that there is also a continuous dependence on $t_{0}$. From this it follows that $\psi^{t}(q,v_{0}):=v(t,q,v_{0})$, where $t \geq 0$, $q \in \mathbb{R}$ and $v_{0} \in \mathbb{H}_{\alpha}$, is a cocycle in $\mathbb{H}_{\alpha}$ over the shift dynamical system $\vartheta^{t}$ on $\mathcal{Q}:=\mathbb{R}$, i.~e. the map $(t,q,v_{0}) \mapsto \psi^{t}(q,v_{0})$ is continuous and the cocycle property
\begin{equation}
	\psi^{t+s}(q,v) = \psi^{t}(\vartheta^{s}(q),\psi^{s}(q,v))
\end{equation}
is satisfied for all $t,s \geq 0$, $q \in \mathcal{Q}$ and $v \in \mathbb{H}_{\alpha}$. Moreover, the map $\psi^{t}(q,\cdot) \colon \mathbb{H}_{\alpha} \to \mathbb{H}_{\alpha}$ is compact for all $t > 0$ and $q \in \mathcal{Q}$. If $F$ is, for example, periodic or almost periodic in $t$, then there are other appropriate choices of $\mathcal{Q}$, which we shall not discuss here (see \cite{Anikushin2020Geom}).

Since we are going to apply the Frequency Theorem to the operator $-A+K+\nu I$ with some $\nu \in \mathbb{R}$, we have to check that such operators posses all the required properties. It is known that unbounded perturbations of a generator of $C_{0}$-semigroup may make the corresponding Cauchy problem incorrect (see \cite{Krein1971}). Here we refer to some specificity of parabolic equations (mainly concerned with sectorial operators), which can be found in the monograph of D.~Henry \cite{Henry1981}.

In the following theorem for convenience we do not mention the complexifications of spaces and operators when this is formally required (especially when dealing with the resolvent). Thus we are making a slight abuse of notation here. In order not to get confused one may think that all the considered spaces are complex.
\begin{theorem}
	\label{TH: ParabLinPertTheorem}
	Under the above assumptions on $A$ and $K$ we have the following:
	\begin{enumerate}
		\item[1)]  The operator $A-K$ with the domain $\mathbb{H}_{1}=\mathcal{D}(A)$ is sectorial and, consequently, there is a $C_{0}$-semigroup (in fact, an analytic semigroup) $G_{K}(t)$ in $\mathbb{H}$ generated by $-A+K$ in $\mathbb{H}=\mathbb{H}_{0}$.
		\item[2)] $G_{K}(t)$ is also a $C_{0}$-semigroup in $\mathbb{H}_{\alpha}$ and satisfies for all $v_{0} \in \mathbb{H}_{\alpha}$ and $t \geq 0$ the identity
		\begin{equation}
			\label{EQ: ParabPerturbedSemigroupDef}
			G_{K}(t)v_{0} = G(t)v_{0} + \int_{0}^{t}G(t-s)KG_{K}(s)v_{0}ds.
		\end{equation}
	    Its generator is the operator $-A+K$ with the domain $\mathbb{H}_{1+\alpha}$.
	    \item[3)]  If for some $\nu \in \mathbb{R}$ the line $-\nu + i\mathbb{R}$ does not intersect with the spectrum of $-A+K$ then the operator $-A+K+\nu I $ satisfies \textbf{(RES)} and we have
	    \begin{equation}
	    	\| (-A+K + (\nu - i \omega)I)^{-1} \|_{\mathbb{H} \to \mathbb{H}_{1}} \to 0 \text{ as } |\omega| \to +\infty,
	    \end{equation}
        where the limit is uniform in $\nu$ from compact subsets of $\mathbb{R}$.
        \item[4)] For any $v_{0} \in \mathbb{H}_{\alpha}$ the mild solution $v(t)=v(t,t_{0},v_{0})$ to \eqref{EQ: ParabolicNonAutEquation} satisfies for all  $t \geq t_{0}$
        \begin{equation}
        	v(t)=G_{K}(t-t_{0})v_{0} + \int_{t_{0}}^{t}G_{K}(t-s)(BF(s,v(s)) + W(s))ds.
        \end{equation}
    	\end{enumerate}
\end{theorem}
\begin{proof}
    1) Put $C_{K}:=\|K\|_{\mathbb{H}_{\alpha} \to \mathbb{H}}$. From the interpolation inequality (see Exercise 4.1.2 in \cite{Chueshov2015} or Theorem 1.4.4 in \cite{Henry1981}) for every $\varepsilon>0$ we have
    \begin{equation}
    	|Kv|_{0} \leq C_{K}|A^{\alpha}v|_{0} \leq C_{K} \varepsilon |Av|_{0} + C_{K} \left(\frac{\alpha}{\varepsilon}\right)^{\alpha/(1-\alpha)} (1-\alpha)|v|_{0}.
    \end{equation}
    Applying Theorem 1.3.2 from \cite{Henry1981}, we get that the operator $A-K$ in $\mathbb{H}$ with the domain $\mathbb{H}_{1}=\mathcal{D}(A)$ is sectorial and, consequently, it generates an analytic semigroup in $\mathbb{H}$ (see Theorem 1.3.4 in \cite{Henry1981}).
    
    2) From the fixed point arguments as in \cite{Chueshov2015} and the smoothing estimate
    \begin{equation}
    	\label{EQ: SmoothingEstimate}
    	\|G(t)\|_{\mathbb{H} \to \mathbb{H}_{\alpha}} \leq \left(\frac{\alpha}{e}\right)^{\alpha} t^{-\alpha} \text{ for all } t > 0.
    \end{equation} 
    we get that the $C_{0}$-semigroup $G_{K}(t)$, $t \geq 0$, in $\mathbb{H}_{\alpha}$ is well-defined by \eqref{EQ: ParabPerturbedSemigroupDef}. The part concerning the generator of this semigroup follows from 1) and Theorem 1.4.8 in \cite{Henry1981}.
    
    3) Note that for any closed operator $M \colon \mathcal{D}(M) \subset \mathbb{H} \to \mathbb{H}$ the map $\varrho(M) \ni p \mapsto (M-pI)^{-1} \in \mathcal{L}(\mathbb{H};\mathcal{D}(M))$, where $\varrho(M)$ denotes the resolvent set of $M$ and the domain $\mathcal{D}(M)$ is endowed with the graph norm, is continuous. Thus, since $\mathbb{H}_{\alpha} \subset \mathbb{H}_{1} = \mathcal{D}(A)$ and the graph norm corresponding to $-A+K$ is equivalent to the norm of $\mathbb{H}_{1}$, for our purposes it is sufficient to show that the value $\|(-A+K-pI)^{-1}\|_{\mathbb{H} \to \mathbb{H}_{1}}$ is bounded uniformly in $p=-\nu + i\omega$ for all sufficiently large $\omega \in \mathbb{R}$. 
    
    Suppose for $w \in \mathbb{H}$ and $v \in \mathbb{H}_{1}$ we have $(-A+K-pI)v=w$. Then $v = (A+pI)^{-1}Kv - (A+pI)^{-1}w$ for $p=-\nu + i\omega$ and sufficiently large $|\omega|$ and we have the estimate
    \begin{equation}
    \label{EQ: ParabTheoremVerify}
    | (-A+K-pI)^{-1}w |_{1} =	|v|_{1} \leq \frac{\| (A+pI)^{-1}\|_{\mathbb{H}\to\mathbb{H}_{1}}}{1 - \| (A+pI)^{-1}\|_{\mathbb{H}\to\mathbb{H}_{1}} \cdot \|K\|_{\mathbb{H}_{\alpha} \to \mathbb{H}}} |w|_{0}.
    \end{equation}
    Clearly, the right-hand side of \eqref{EQ: ParabTheoremVerify} tends to $0$ as $|\omega| \to +\infty$ uniformly in $\nu$ from compact subsets of $\mathbb{R}$.
    
    4) For $v_{0} \in \mathcal{D}(A)$ and $\xi(\cdot) \in C^{1}([0,T];\mathbb{H})$ let us consider the mild solution $v(\cdot)$ to $\dot{v}(t)=(-A+K)v(t) + \xi(t)$, which is given by
    \begin{equation}
    	v(t)=G_{K}(t)v_{0} + \int_{0}^{t}G_{K}(t-s)\xi(s)ds.
    \end{equation}
    Since $v_{0} \in \mathcal{D}(A)$ and $ \mathcal{D}(A) = \mathcal{D}(-A+K)$ by item 1), we in fact have $v(\cdot) \in C^{1}([0,T];\mathbb{H})$, $v(t) \in \mathcal{D}(A)$ and $\dot{v}(t)=(-A+K)v(t)+\xi(t)=-Av(t)+Kv(t)+\xi(t)$ for all $t \in [0,T]$ (see Theorem 6.5 in \cite{Krein1971}). Thus for all $v_{0} \in \mathcal{D}(A)$ and $\xi(\cdot) \in C^{1}([0,T];\mathbb{H})$ we have the equality
    \begin{equation}
    	\label{EQ: ParabPertTheoremLast}
    	G(t)v_{0}+\int_{0}^{t}G(t-s)(Kv(s)+\xi(s))ds = G_{K}(t)v_{0} + \int_{0}^{t}G_{K}(t-s)\xi(s)ds.
    \end{equation}
    Since $\mathcal{D}(A)$ is dense in $\mathbb{H}_{\alpha}$ and $C^{1}([0,T];\mathbb{H})$ is dense in $C([0,T];\mathbb{H})$, we can extend \eqref{EQ: ParabPertTheoremLast} for $v_{0} \in \mathbb{H}_{\alpha}$ and $\xi(\cdot) \in C([0,T];\mathbb{H})$ such that $\xi(s)=BF(s,Cv(s))+W(s)$, where $v(s)=v(s,0,v_{0})$ is a mild solution to \eqref{EQ: ParabolicNonAutEquation}. Indeed, from the fixed point arguments we get that solutions to
    \begin{equation}
    	v(t)=G(t)v_{0} + \int_{0}^{t}G(t-s)(Kv(s)+\xi(s))ds
    \end{equation}
    depend continuously on $\xi(\cdot) \in C([0,T];\mathbb{H})$ in the space $C([0,T];\mathbb{H}_{\alpha})$. The theorem is proved.
\end{proof}

So, if $w(t)=v(t,t_{0},v_{0,1})-v(t,t_{0},v_{0,2})$ is the difference of two mild solutions of \eqref{EQ: ParabolicNonAutEquation} corresponding to $v_{0,1},v_{0,2} \in \mathbb{H}_{\alpha}$, then it is also a mild solution to the inhomogeneous linear equation
\begin{equation}
\label{EQ: ParabInhomogeneous}
\dot{w}(t)=(-A+K)w(t)+B\xi(t),
\end{equation}
where $\xi(t):=F(t,Cv(t,t_{0},v_{0,1}))-F(t,Cv(t,t_{0},v_{0,2}))$ and $w(t_{0})=v_{0,1}-v_{0,2} \in \mathbb{H}_{\alpha}$. Note that $w(\cdot) \in C([0,T];\mathbb{H}_{\alpha})$ for all $T > 0$.

Now let us consider the quadratic form $\mathcal{F}(v,\xi)$ for $v \in \mathbb{H}_{\alpha}$ and $\xi \in \Xi$ as in \eqref{EQ: RealQuadraticFormParab}. We will also require that $\mathcal{F}(\cdot,0) \geq 0$ and for any $v_{1},v_{2} \in \mathbb{H}_{\alpha}$ we have $\mathcal{F}(v_{1}-v_{2},\xi_{1}-\xi_{2}) \geq 0$ provided that $\xi_{i} = F(t,Cv_{i})$, where $i=1,2$ and $t \in \mathbb{R}$ is arbitrary. Now we have to apply Theorem \ref{TH: ParabRealFreqTheorem} with some $\nu \in \mathbb{R}$ to the pair $(-A+K + \nu I, B)$. For this we have to check the frequency-domain condition
\begin{equation}
\label{EQ: FreqConditionFormGeneral}
\sup_{\xi \in \Xi^{\mathbb{C}}, \xi \not = 0} \frac{\mathcal{F}^{\mathbb{C}}(-(-A+K-pI)^{-1}B\xi,\xi)}{|\xi|^{2}_{\Xi^{\mathbb{C}}}} <  0 \text{ for all } p = -\nu + i \omega, \text{ where } \omega \in \mathbb{R}. 
\end{equation}

As a special case let us consider $\mathcal{F}(v,\xi):= \Lambda^{2}|Cv|^{2}_{\mathbb{M}}-|\xi|^{2}_{\Xi}$ for $v \in \mathbb{H}_{\alpha}$ and $\xi \in \Xi$. It is clear that $\mathcal{F}(\cdot,0) \geq 0$ and we have $\mathcal{F}(v_{1}-v_{2},\xi_{1}-\xi_{2}) \geq 0$ for any $v_{1},v_{2} \in \mathbb{H}_{\alpha}$ and $\xi_{i} = F(t,Cv_{i})$, where $i=1,2$ and $t \in \mathbb{R}$ is arbitrary. Note also that since $(Cv,Cv)_{\mathbb{M}} = (C^{*}Cv,v)_{\alpha}$, the Hermitian extension of $\mathcal{F}$ has the form as in \eqref{EQ: QuadraticFormComplex}. If we consider the transfer operator $W(p):=C(-A+K-pI)^{-1}B$, then \eqref{EQ: FreqConditionFormGeneral} in this case is equivalent to
\begin{equation}
\label{EQ: FreqParabGeneral}
\|W(-\nu + i\omega)\|_{\Xi^{\mathbb{C}} \to \mathbb{M}^{\mathbb{C}}} < \Lambda^{-1}, \text{ for all } \omega \in \mathbb{R}.
\end{equation}

If for some $\nu \in \mathbb{R}$ the $C_{0}$-semigroup $e^{\nu t}G_{K}(t)$ (which is generated by $-A+K+\nu I$) in $\mathbb{H}_{\alpha}$ admits an exponential dichotomy, we denote the stable and unstable spaces (w.~r.~t. this dichotomy) by $\mathbb{H}^{s}_{\alpha}(\nu)$ and $\mathbb{H}^{u}_{\alpha}(\nu)$ respectively. Note that in virtue of item 1) of Theorem \ref{TH: ParabLinPertTheorem} we have that the generator of $G_{K}(t)$ in $\mathbb{H}_{\alpha}$ is given by the operator $-A+K$ with the domain $\mathbb{H}_{1+\alpha}$. Since the inclusion $\mathbb{H}_{1+\alpha} \subset \mathbb{H}_{\alpha}$ is compact, the generator has compact resolvent and its spectrum is discrete \cite{EngelNagel2000}. Thus the $C_{0}$-semigroup $e^{\nu t}G_{K}(t)$ in $\mathbb{H}_{\alpha}$ admits an exponential dichotomy with the unstable space $\mathbb{H}^{u}_{\alpha}$ such that $\operatorname{dim} \mathbb{H}^{u}_{\alpha} = j$ if and only if for the operator $-A+K$ there are exactly $j \geq 0$ eigenvalues with $\operatorname{Re}\lambda > -\nu$ and no eigenvalues lie on the line $-\nu + i \mathbb{R}$ (see \cite{EngelNagel2000}).

\begin{theorem}
	\label{TH: ParabEqsMainAssumptionsVer}
	Suppose there is some $\nu \in \mathbb{R}$ such that the semigroup $e^{\nu t}G_{K}(t)$ in $\mathbb{H}_{\alpha}$ admits an exponential dichotomy with the subspaces $\mathbb{H}^{s}_{\alpha}(\nu)$ and $\mathbb{H}^{u}_{\alpha}(\nu)$ such that $\operatorname{dim}\mathbb{H}^{u}_{\alpha} = j \geq 0$. Let the frequency inequality in \eqref{EQ: FreqConditionFormGeneral} be satisfied. Then for the cocycle $(\psi,\vartheta)$ generated by \eqref{EQ: ParabolicNonAutEquation} in $\mathbb{H}_{\alpha}$ we have 
	\begin{description}
		\item[\textbf{(H1)}] There exists an operator $P \in \mathcal{L}(\mathbb{H}_{\alpha})$, self-adjoint in $\mathbb{H}_{\alpha}$ and a decomposition of $\mathbb{H}_{\alpha}$ into the direct sum $\mathbb{H}_{\alpha}=\mathbb{H}^{+}_{\alpha} \oplus \mathbb{H}^{-}_{\alpha}$ such that $P$ is positive on $\mathbb{H}^{+}_{\alpha}$ and negative on $\mathbb{H}^{-}_{\alpha}$. One may take $\mathbb{H}^{+}_{\alpha} = \mathbb{H}^{s}_{\alpha}(\nu)$ and $\mathbb{H}^{-}_{\alpha} = \mathbb{H}^{u}_{\alpha}(\nu)$.
		\item[\textbf{(H2)}] We have $\operatorname{dim}\mathbb{H}^{-}_{\alpha}=j$.
		\item[\textbf{(H3)}] For the quadratic form $V(v):=(Pv,v)_{\alpha}$ and some $\delta>0$ we have
		\begin{equation}
		\begin{split}
		e^{2\nu r} V(\psi^{r}(q,v_{1})-\psi^{r}(q,v_{2})) - e^{2\nu l} V(\psi^{l}(q,v_{1})-\psi^{l}(q,v_{2})) \leq \\ \leq -\delta \int_{l}^{r}e^{2\nu s}|\psi^{s}(q,v_{1})-\psi^{s}(q,v_{2})|^{2}_{\alpha}ds
		\end{split}
		\end{equation}
		satisfied for all $q \in \mathcal{Q}$, $v_{1},v_{2} \in \mathbb{H}_{\alpha}$ and $0 \leq l \leq r$.
	\end{description}
\end{theorem}
\begin{proof}
	Theorem \ref{TH: ParabLinPertTheorem} allows us apply Theorem \ref{TH: ParabRealFreqTheorem} to the pair $(-A+K+\nu I, B)$ and the form $\mathcal{F}$. Thus, there exists an operator $P \in \mathcal{L}(\mathbb{H}_{\alpha})$, self-adjoint in $\mathbb{H}_{\alpha}$ and such that
	\begin{equation}
		\label{EQ: ParabApplicationFreqThFirstIneq}
		e^{2\nu r} V(w(r)) - V(w(0)) + \int_{0}^{r}e^{2\nu s}\mathcal{F}(w(s),\xi(s))ds \leq -\delta \int_{0}^{r}e^{2\nu s} |w(s)|^{2}_{\alpha}ds
	\end{equation}
    holds for any solution $w(\cdot)=w(\cdot,w_{0},\xi)$ to $\dot{w}(t)=(-A+K)w(t) + B\xi(t)$ such that $\xi(\cdot) \in L_{2}(0,T;\Xi)$ and $w(\cdot) \in C([0,T];\mathbb{H}_{\alpha})$. 
    
    Let $w(t)=v_{1}(t)-v_{2}(t)$ be the difference of two mild solutions $v_{1}(t) = v(t,0,v_{0,1})$ and $v_{2}(t) = v(t,0,v_{0,2})$ of \eqref{EQ: ParabolicNonAutEquation} corresponding to some $v_{0,1},v_{0,2} \in \mathbb{H}_{\alpha}$ and $\xi(t) := F(t,Cv_{1}(t))-F(t,Cv_{2}(t))$. Then we have $F(w(s),\xi(s)) \geq 0$ for all $s \geq 0$ and, consequently,
    \begin{equation}
    	\label{EQ: ParabApplicationFreqThSecondIneq}
    	e^{2\nu r} V(v_{1}(r)-v_{2}(r)) - V(v_{1}(0)-v_{2}(0)) \leq -\delta \int_{0}^{r}e^{2\nu s} |v_{1}(s)-v_{2}(s)|^{2}_{\alpha}ds.
    \end{equation}
    From this and the cocycle property, we have \textbf{(H3)} satisfied.
    
    Now let $\mathbb{H}^{+}_{\alpha}:=\mathbb{H}^{s}_{\alpha}(\nu)$ and $\mathbb{H}^{-}_{\alpha}:=\mathbb{H}^{u}_{\alpha}(\nu)$. By our assumptions, we have $\operatorname{dim}\mathbb{H}^{-}_{\alpha}=j$. Since $\mathcal{F}(w(\cdot),0) \geq 0$, from \eqref{EQ: ParabApplicationFreqThFirstIneq} we also have
    \begin{equation}
    	V(v(T))-V(v(0)) \leq -\delta \int_{0}^{T}|v(s)|^{2}_{\alpha} ds
    \end{equation}
    for any $T>0$ and any solution $v(\cdot)$ of the linear equation $\dot{v}(t)=(-A+K+\nu I)v(t)$ with $v(0) \in \mathbb{H}_{\alpha}$. Now we can apply Theorem \ref{TH: HighRankConesTheorem} to get the required sign properties of $P$ on $\mathbb{H}^{+}_{\alpha}$ and $\mathbb{H}^{-}_{\alpha}$. The proof is finished.
\end{proof}

In the case $\nu > 0$ the properties established in Theorem \ref{TH: ParabEqsMainAssumptionsVer} along with the mentioned Lipschitzity and compactness properties guarantee (if there exists at least one bounded on $\mathbb{R}$ solution) the existence of an invariant family of $j$-dimensional Lipschitz submanifolds in $\mathbb{H}_{\alpha}$, which forms the inertial manifold for the cocycle (see \cite{Anikushin2020Red} and \cite{Anikushin2020Geom}). Properties of these manifolds such as Lipschitzity, exponential tracking (the rate is determined by the exponent $\nu >0$), $C^{1}$-differentiability and normal hyperbolicity can be also established with the aid of abstract quadratic functionals (see \cite{Anikushin2020Geom}). 

It is known that the $C^{1}$-differentiability of inertial manifolds is the maximum\footnote{Note that in some cases \cite{Zelik2014,KostiankoZelikSA2020} it is also possible to obtain the $C^{1+\varepsilon}$-differentiability for a small $\varepsilon > 0$.} that can be obtained under the Spectral Gap Condition (see Subsection \ref{SUBSEC: ParabSpectralGapCond}) and for differentiability of higher orders stronger assumptions are required \cite{ChowLuSell1992}. We do not know how to express these stronger assumptions in terms of \eqref{EQ: FreqParabGeneral}, \eqref{EQ: FreqConditionFormGeneral} or \textbf{(H3)}.

\subsection{\textbf{Optimality of the frequency inequality}}
\label{SUBSEC: OptimalityParab}
The following theorem follows from the so-called theorem on the losslessness of the $S$-procedure \cite{Gelig1978}. It states, roughly speaking, that after the passing from \eqref{EQ: ParabApplicationFreqThFirstIneq} to \eqref{EQ: ParabApplicationFreqThSecondIneq} by eliminating $\mathcal{F}$, we do not loss much information so the frequency inequality remains optimal (in the class of systems) even for condition \textbf{(H3)} obtained after such an elimination. For simplicity and needs for further discussions, we state it for the frequency inequality from \eqref{EQ: FreqParabGeneral}.
\begin{theorem}
	\label{TH: ParabOptimalityTh}
	Let $-A \colon \mathbb{H} \to \mathbb{H}$ be a sectorial operator with compact resolvent, $B \colon \Xi \to \mathbb{H}$ and $C \colon \mathbb{H}_{\alpha} \to \mathbb{M}$ be bounded. Let $\Lambda>0$ be fixed and $\nu \in \mathbb{R}$ be such that $A$ does not have spectrum on the line $-\nu + i \mathbb{R}$. Suppose also that the subspace $\Xi_{0} := B^{-1}(\mathbb{H}_{\alpha})$ is dense in $\Xi$. Then the frequency inequality 
	\begin{equation}
		\label{EQ: FreqIneqSimple}
		\|C(A-pI)^{-1}B\|_{\Xi^{\mathbb{C}} \to \mathbb{M}^{\mathbb{C}}} < \Lambda^{-1} \text{ for all } p=-\nu + i\omega \text{, where } \omega \in \mathbb{R}
	\end{equation}
    is equivalent to the existence of $P \in \mathcal{L}(\mathbb{H}_{\alpha})$, which is self-adjoint in $\mathbb{H}_{\alpha}$, and $\delta>0$ such that the inequality
    \begin{equation}
    	(Pv,(A+\nu I)v+B\xi)_{\alpha} \leq -\delta |v|^{2}_{\alpha}
    \end{equation}
    is satisfied for all $v \in \mathbb{H}_{1+\alpha}$ and $\xi \in \Xi_{0}$ such that $|\xi|^{2}_{\Xi} \leq \Lambda^{2} |C v^{2}|_{\alpha}$.
\end{theorem}
\begin{proof}
Let us put $\mathcal{G}_{0}(v,\xi) := 2(Pv,(A+\nu I)v+B\xi)_{\alpha} + \delta (|v|^{2}_{\alpha} + |\xi|^{2}_{\Xi})$ and $\mathcal{G}(v,\xi) = |\xi|^{2}_{\Xi} - \Lambda^{2}|C v|^{2}_{\alpha}$. Then the problem is whether the conditions
\begin{description}
	\item[\textbf{(G1)}] $\mathcal{G}_{0}(v,\xi) \leq 0$ for all $v \in \mathbb{H}_{1+\alpha}$, $\xi \in \Xi_{0}$ such that $\mathcal{G}(v,\xi) \leq 0$.
	\item[\textbf{(G2)}] There exists $\tau \geq 0$ such that $\mathcal{G}_{0}(v,\xi) - \tau \mathcal{G}(v,\xi) \leq 0$ for all $v \in \mathbb{H}_{1+\alpha}$, $\xi \in \Xi_{0}$.
\end{description}
are equivalent. Indeed, by Theorem \ref{TH: ParabRealFreqTheorem} (applied to the pair $(A+\nu I, B)$ and the form $\mathcal{F} = -\mathcal{G}$) the frequency inequality from \eqref{EQ: FreqIneqSimple} is equivalent to \textbf{(G2)} with $\tau=1$. To see this one should consider \eqref{EQ: RealFreqThInequality} for $\xi(\cdot) \equiv \xi_{0} \in \Xi_{0}$ and $v(\cdot)=v(\cdot,v_{0},\xi)$ with $v_{0} \in \mathbb{H}_{1+\alpha}$. Then $v(\cdot)$ is a continuously differentiable $\mathbb{H}_{\alpha}$-valued function and \textbf{(G2)} with $\tau = 1$ follows after dividing \eqref{EQ: RealFreqThInequality} by $T>0$ and taking it to the limit as $T \to 0+$. The necessity of the frequency inequality \eqref{EQ: FreqIneqSimple} for \textbf{(G2)} follows by considering $v=-(A+ (\nu - i\omega)I)^{-1}B\xi$ for $\xi \in \Xi^{\mathbb{C}}_{0}$ (after taking complexifications).

It is clear that from \textbf{(G2)} we immediately get \textbf{(G1)}. So, the only non-trivial part is how to get \textbf{(G2)} from \textbf{(G1)}. This follows\footnote{Note that the proof is based on the fact that the image of $\mathbb{H}_{1+\alpha} \times \Xi_{0}$ in $\mathbb{R}^{2}$ under the map $(v,\xi) \mapsto (\mathcal{G}(v,\xi),\mathcal{G}_{0}(v,\xi))$ is convex (a theorem of L.~L.~Dines). This fact is purely algebraic and require neither completeness of the spaces nor boundedness of the quadratic forms.} from Theorem 1 in \cite{Yakubovich1973Sproc} or Theorem 2.17 in \cite{Gelig1978} (although the latter is stated for Euclidean spaces only, the used arguments can be easily adapted for our case). So, the proof is finished.
\end{proof}

An immediate corollary of Theorem \ref{TH: ParabOptimalityTh} shows that \eqref{EQ: FreqIneqSimple} is necessary and sufficient for the existence of a common quadratic functional in the class of linear perturbations as follows. Of course, the only interesting part is the necessity.
\begin{corollary}
	In terms of Theorem \ref{TH: ParabOptimalityTh} the frequency inequality \eqref{EQ: FreqIneqSimple} is necessary and sufficient for the existence of $P \in \mathcal{L}(\mathbb{H}_{\alpha})$, which is self-adjoint in $\mathbb{H}_{\alpha}$, and $\delta>0$ such that the inequality
	\begin{equation}
		(Pv,(A + BMC + \nu I)v)_{\alpha} \leq -\delta |v|^{2}_{\alpha}
	\end{equation}
    is satisfied for any $v \in \mathbb{H}_{1+\alpha}$ and any $M \in \mathcal{L}(\mathbb{M};\Xi)$ with $\| M \| \leq \Lambda$.
\end{corollary}
\begin{proof}
	Let $\mathcal{S}_{1}$ (resp. $\mathcal{S}_{2}$) denote the sets of pairs $(v,\xi) \in \mathbb{H}_{1+\alpha} \times \Xi_{0}$ such that $|\xi|_{\Xi} \leq \Lambda |Cv|_{\alpha}$ (resp. $\xi = M C v$ for some $M \in \mathcal{L}(\mathbb{M};\Xi)$ with $\|M\| \leq \Lambda$). Note that since $-A-BMC$ is sectorial (see item 2 of Theorem \ref{TH: ParabLinPertTheorem}), we have $BMCv \in \mathbb{H}_{\alpha}$ and, consequently, $MC v \in \Xi_{0}$ whenever $v \in \mathbb{H}_{1+\alpha}$. In virtue of Theorem \ref{TH: ParabOptimalityTh}, for our purposes it is sufficient to show that $\mathcal{S}_{1} = \mathcal{S}_{2}$. It is clear that $\mathcal{S}_{2} \subset \mathcal{S}_{1}$. Let $(v,\xi) \in \mathcal{S}_{1}$. For $y \in \mathbb{M}$ we define $M \in \mathcal{L}(\mathbb{M};\Xi)$ as
	\begin{equation}
		My := \xi\frac{(Cv,y)_{\mathbb{M}}}{(Cv,Cv)_{\mathbb{M}}}.
	\end{equation}
    It is clear that $\xi = MCv$ and $\|M\| \leq \Lambda$. Thus, $\mathcal{S}_{1} = \mathcal{S}_{2}$ and the proof is finished.
\end{proof}
So, in terms of semi-dichotomies, the frequency inequality \eqref{EQ: FreqIneqSimple} is equivalent to that any system in the class of linear perturbations admits a semi-dichotomy, which is in some sense ``strongly uniform'' in the class (see Remark \ref{REM: UniformAndKalman}). It is interesting when such a kind of uniformity is equivalent to its ``nonuniform'' analog as in the following problem.
\begin{problem}
	\label{PROB: NonuniformSemiDichotomies}
	In terms of Theorem \ref{TH: ParabOptimalityTh} find conditions on $A,B,C$ such that the frequency inequality \eqref{EQ: FreqIneqSimple} is necessary for the property that any operator $A+BMC$, where $M \in \mathcal{L}(\mathbb{M};\Xi)$ with $\| M \| \leq \Lambda$ is arbitrary, has the same number of eigenvalues as $A$ located to the right of the line $-\nu +i \mathbb{R}$ and has no eigenvalues on this line.
\end{problem}
Usually, optimal conditions for the existence of inertial manifolds are obtained as a positive solution to Problem \ref{PROB: NonuniformSemiDichotomies} in the considered class of equations. This includes the classical case when $A$ is self-adjoint and $B=C=I$ (with the spaces $\Xi=\mathbb{M}= \mathbb{H}$) studied by A.~V.~Romanov in \cite{Romanov1994}, where, however, the optimality is understood in a bit weaker sense since in \cite{Romanov1994} we are also allowed to vary the linear operator $A$ (it is not hard to see that we may keep it fixed). Another optimality result was recently obtained by A.~Kostianko and S.~Zelik in \cite{KostiankoZelik2019Kwak}, where the frequency inequality \eqref{EQ: FreqIneqSimple} is verified for a special class of non-self-adjoint operators $A$ obtained trough the so-called Kwak transform and it is also shown that Problem \ref{PROB: NonuniformSemiDichotomies} has a positive solution in this class. Note also that in \cite{KostiankoZelik2019Kwak} the operators $B$ and $C$ are non-trivial.  It seems also that the optimality result obtained in V.~V.~Chepyzhov, A.~Kostianko and S.~Zelik \cite{ChepyzhovKostiankoZelik2019} for hyperbolic relaxations of an abstract semilinear parabolic equation also gives positive solution to an analog of Problem \ref{PROB: NonuniformSemiDichotomies} and follows the same scheme as \cite{KostiankoZelik2019Kwak} by verifying the frequency inequality \eqref{EQ: FreqIneqSimple} at first\footnote{As we have noted in the introduction our abstract version of the Frequency Theorem is also applicable for hyperbolic equations since it only assumes the $C_{0}$-semigroup property. But we cannot study inertial manifolds for such equations by means of our theory \cite{Anikushin2020Geom} due to the compactness assumption.}.

\begin{remark}
\label{REM: UniformAndKalman}
In general, the conclusion of Problem \ref{PROB: NonuniformSemiDichotomies} for general operators $A,B,C$ fails to hold. Counterexamples arise even in finite-dimensions as counterexamples to the Kalman conjecture in stability theory (see, the review of G.~A.~Leonov and N.~V.~Kuznetsov \cite{LeoKuz2013Hidden} or I.~M.~Boiko et. al \cite{BoikoKuzetall2020}). Such counterexamples show that there exist systems such that $A+BMC$ has only eigenvalues with negative real parts for any $\|M\| \leq \Lambda$, but the frequency inequality is not satisfied. Indeed, if the frequency inequality is satisfied, then we would have a common positive-definite quadratic functional\footnote{In finite-dimensions the existence of a common positive-definite quadratic Lyapunov functional is equivalent to the existence of a common ellipsoid such that trajectories of any system cross it strictly inwards.} for the class, which includes nonlinear systems also. This implies that the zero solution is globally exponentially stable for all the nonlinear problems (we assume $F(0)=0$). But in the counterexamples there exist nonlinearities from the class such that the corresponding system has periodic orbits. Thus, in the case when there is no uniformity given by the frequency inequality, the overall dynamics may strongly depend on the choice of the nonlinearity $F$ and, consequently, the problem for the class of perturbations become unsolvable.

So, in general we have only the ``uniform'' optimality stated in Theorem \ref{TH: ParabOptimalityTh} and when this uniformity is broken, we cannot guarantee anything common for dynamics in the class of systems.
\end{remark}

Now we are going to discuss some particular cases of the frequency-domain condition in \eqref{EQ: FreqParabGeneral} and its modifications.

\subsection{\textbf{Spectral Gap Condition}}
\label{SUBSEC: ParabSpectralGapCond}
Let us consider the special case of $K=0$; $\Xi := \mathbb{H}_{\beta}$ for some $\beta \in [0,\alpha]$; $\mathbb{M}:=\mathbb{H}_{\alpha}$ and $C$, $B$ being the identity operators. As above, let $0 < \lambda_{1} \leq \lambda_{2} \leq \ldots$ be the eigenvalues of $A$. Let us fix $j$ such that $\lambda_{j+1}-\lambda_{j}>0$ and find $\nu \in (\lambda_{j},\lambda_{j+1})$ such that
\begin{equation}
\label{EQ: FrequencyConditionParabBefore}
\| (A+(-\nu + i \omega)I)^{-1} \|_{\mathbb{H}_{\beta} \to \mathbb{H}_{\alpha}} < \Lambda^{-1}, \text{ for all } \omega \in \mathbb{R}.
\end{equation}
Using the orthogonal basis corresponding to $\lambda_{k}$'s, one can easily show that
\begin{equation}
\| (A - (\nu + i \omega)I)^{-1} \|_{\mathbb{H}_{\beta} \to \mathbb{H}_{\alpha}} = \sup_{k} \frac{\lambda^{\alpha-\beta}_{k} }{|\lambda_{k} - \nu - i\omega|} \leq \sup_{k}\frac{\lambda^{\alpha-\beta}_{k} }{|\lambda_{k} - \nu|}.
\end{equation}
Moreover, due to some monotonicity, we have
\begin{equation}
\label{EQ: ParabSpectralNormMax}
\sup_{k}\frac{\lambda^{\alpha-\beta}_{k}}{ |\lambda_{k} - \nu| } = \max\left\{ \frac{\lambda^{\alpha-\beta}_{j}}{\nu - \lambda_{j}}, \frac{\lambda^{\alpha-\beta}_{j+1}}{\lambda_{j+1} - \nu} \right\}
\end{equation}
and the norm will be the smallest possible if the values under the maximum in \eqref{EQ: ParabSpectralNormMax} coincide. So, for this one should take
\begin{equation}
\nu = \frac{\lambda^{\alpha-\beta}_{j+1}}{\lambda^{\alpha-\beta}_{j}+\lambda^{\alpha-\beta}_{j+1}} \cdot \lambda_{j} + \frac{\lambda^{\alpha-\beta}_{j}}{\lambda^{\alpha-\beta}_{j}+\lambda^{\alpha-\beta}_{j+1}} \cdot \lambda_{j+1} \in ( \lambda_{j},\lambda_{j+1} ).
\end{equation}
For such a choice of $\nu$ the frequency condition \eqref{EQ: FrequencyConditionParabBefore} takes the form
\begin{equation}
\label{EQ: SpectralGapConditionParab}
\frac{ \lambda_{j+1} - \lambda_{j} }{\lambda^{\alpha-\beta}_{j} + \lambda^{\alpha-\beta}_{j+1} } > \Lambda,
\end{equation}
known as the Spectral Gap Condition, which guarantees the existence of $j$-dimensional inertial manifolds for semilinear parabolic equations. We should immediately note here that the pioneering paper of C.~Foias, G.~R.~Sell and R.~Temam \cite{FoiasSellTemam1988} and some other works (for example, G.~R.~Sell and Y.~You\cite{SellYou2002}) contain non-optimal constants in this condition. Instead of the Lipschitz constant $\Lambda$, there are various (sometimes unefficient) constants, which appear as artifacts of the fixed point method and rough estimates. The first optimal versions appeared in the works of M.~Miklav\v{c}i\v{c} \cite{Miclavcic1991} and A.~V.~Romanov \cite{Romanov1994}.

\subsection{\textbf{Other frequency-domain conditions}}
\label{SUBSEC: ParabOtherFreqIneq}

For $\Xi = \mathbb{M} = \mathbb{H}$ and $B = I$ the frequency inequality \eqref{EQ: FreqParabGeneral} was used in the work of M.~Miklav\v{c}i\v{c} \cite{Miclavcic1991}, although his method seems to work in the case of non-trivial $B$ as well. Moreover, he also used a more general condition obtained through the inequality
\begin{equation}
	|F(t,y_{1})-F(t,y_{2})|_{\mathbb{H}} \leq \sum_{k=1}^{m} |C_{k}(y_{1}-y_{2})|_{\mathbb{H}} \text{ for all } y_{1},y_{2} \in \mathbb{H}_{\alpha} \text{ and } t \in \mathbb{R}
\end{equation}
for some operators $C_{1},\ldots,C_{m} \in \mathcal{L}(\mathbb{H}_{\alpha};\mathbb{H})$. For $m>1$ this inequality is not ``quadratic'' and, consequently, we cannot study the corresponding equation via the Frequency Theorem. However, if we change it to
\begin{equation}
	|F(y_{1})-F(y_{2})|^{2}_{\mathbb{H}} \leq \sum_{k=1}^{m} |C_{k}(y_{1}-y_{2})|^{2}_{\mathbb{H}} \text{ for all } y_{1},y_{2} \in \mathbb{H}_{\alpha} \text{ and } t \in \mathbb{R},
\end{equation}
we immediately get the quadratic form $\mathcal{F}(v,\xi) = \sum_{k=1}^{m}|C_{k}v|^{2}_{\mathbb{H}} - |\xi|^{2}_{\mathbb{H}}$ and the corresponding to it frequency inequality \eqref{EQ: FreqConditionFormGeneral}.

Somewhat weaker (non-optimal) conditions for the non-self-adjoint case were also obtained by G.~R.~Sell and Y.~You in \cite{SellYou1992}.

In fact, condition \eqref{EQ: SpectralGapConditionParab} is too rough to obtain concrete results in some situations. The presence in \eqref{EQ: FreqParabGeneral} of the measurement operator $C$ and of the control operator $B$ as well as consideration of the possibly non-self adjoint operator $-A+K$ in the linear part give a lot of flexibility. This intuition, which came from control theory, allowed R.~A.~Smith to obtain non-trivial results (especially concerned with generalizations of the Poincar\'{e}-Bendixson theory) by constructing $2$-dimensional inertial manifolds for certain ODEs, delay equations and parabolic problems. See, for example, \cite{Smith1994PB1,Smith1994PB2}, where applications to the FitzHugh-Nagumo system and the diffusive Goodwin system are given. However, he seemed to be unfamiliar with the Frequency Theorem and his methods were based on a priori estimates. This forced him to be restricted to concrete classes of systems. In particular, in \cite{Smith1994PB1,Smith1994PB2} only the case of reaction-diffusion equations ($\alpha = 0$) is treated and, moreover, there only domains in dimensions $2$ and $3$ are considered. Our approach, which is based on the use of quadratic Lyapunov functionals, does not have such limitations. Such a generalization of Smith's approach is presented by the author in \cite{Anik2020PB}.

\subsection{\textbf{The Circle Criterion}} Suppose the measurement operator $C$ is a linear functional (i.~e. $\mathbb{M} = \mathbb{R}$), $\Xi = \mathbb{R}$ and for some constants $\varkappa_{1} \leq \varkappa_{2}$ we have
\begin{equation}
	\varkappa_{1} \leq \frac{F(t,y_{1})-F(t,y_{2})}{y_{1}-y_{2}} \leq \varkappa_{2}
\end{equation}
satisfied for all $t \in \mathbb{R}$ and $y_{1},y_{2} \in \mathbb{R}$ such that $y_{1} \not= y_{2}$. Then an appropriate choice of the form $\mathcal{F}$ is $\mathcal{F}(v,\xi):=(\xi- \varkappa_{1} Cv) (\varkappa_{2}Cv - \xi)$. The corresponding frequency-domain condition takes the form
\begin{equation}
	\label{EQ: FreqCircleCriterionParab}
	\operatorname{Re}\left[(1+\varkappa_{1}W(p))^{*}(1+\varkappa_{2}W(p))\right] > 0 \text{ for } p= -\nu + i\omega, \text{ where } \omega \in \mathbb{R},
\end{equation}
which is known in the control theory as the Circle Criterion \cite{Gelig1978}. In the case $\varkappa_{1} > 0$ we no longer have the property $\mathcal{F}(v,0) \geq 0$, which was used in Theorem \ref{TH: ParabEqsMainAssumptionsVer} in order to establish the sign properties of the operator $P$ from the dichotomy properties of $-A+K+\nu I$. In this case one should use dichotomy properties of the operator $-A+K+\nu I + \varkappa_{0} BC$ for some $\varkappa_{0} \in [\varkappa_{1},\varkappa_{2}]$ since we have $\mathcal{F}(v,\varkappa_{0}Cv) \geq 0$. The circle criterion also gives some flexibility in applications, if it can be applied \cite{Anikushin2020Red,Anikushin2020FreqDelay,Gelig1978}.

\subsection{\textbf{An example}}
As an illustrative example we consider the following system with homogeneous Dirichlet boundary conditions given by
\begin{equation}
	\label{EQ: ParabNeumannExample1}
	\begin{split}
		u_{t}(t,x)&=u_{xx}(t,x) + w(t,x) + f_{1}(u(t,x)), x \in (0,1), t > 0, \\
		w_{t}(t,x)&=w_{xx}(t,x) - w(t,x) + f_{2}(u(t,x)), x \in (0,1), t > 0, \\
		u(t,0)&=u(t,1)=w(t,0)=w(t,1)=0, t > 0\\
		u(0,\cdot)&=u_{0}(\cdot) \in L_{2}(0,1),\\
		w(0,\cdot)&=w_{0}(\cdot) \in L_{2}(0,1).
	\end{split}
\end{equation}
It clearly can be written in the abstract form \eqref{EQ: ParabolicNonAutEquation} for $\alpha = 0$ with $\mathbb{H} = L_{2}(0,1; \mathbb{R}^{2})$; $-Av = (u_{xx} + w, w_{xx} - w)$ for $v=(u,w) \in \mathcal{D}(A) = W^{2,2}(0,1;\mathbb{R}^{2}) \cap H^{1}_{0}(0,1;\mathbb{R}^{2})$; $\Xi = L_{2}(0,1;\mathbb{R}^{2})$ and $B = I$; $\mathbb{M} = L_{2}(0,1)$ with $Cv = u$ for any $v=(u,w) \in \mathbb{H}$; $F(v)=(f_{1}(u),f_{2}(u))$ for $v=(u,w) \in \mathbb{H}$ and $W \equiv 0$.

Here we will show how to obtain conditions, which utilizes the specific structure of $F$ and which are sharper than the usual Spectral Gap Condition. For this we suppose that $f_{1}$ and $f_{2}$ are continuously differentiable and some constants $\mu_{1},\mu_{2},\varkappa_{1},\varkappa_{2} > 0$ we have $-\mu_{1} \leq f_{1}'(u) \leq \mu_{2}$ and $-\varkappa_{1} \leq f'_{2}(u) \leq \varkappa_{2}$ for all $u \in \mathbb{R}$.

Let us consider the quadratic form in $v=(u,w) \in \mathbb{H}$ and $\xi = (\xi_{1},\xi_{2}) \in \Xi$ such as
\begin{equation}
	\label{EQ: ParabFormExample}
	\begin{split}
			\mathcal{F}(u,w,\xi_{1},\xi_{2}) = \tau_{1}\int_{0}^{1}(\mu_{2}u(x)-\xi_{1}(x))(\xi_{1}(x)+\mu_{1}u(x) )dx +\\+ \tau_{2}\int_{0}^{1}(\varkappa_{2}u(x)-\xi_{2}(x))(\xi_{2}(x)+\varkappa_{1}u(x) )dx,
	\end{split}
\end{equation}
where $\tau_{1},\tau_{2} \geq 0$ are parameters. It is clear that it satisfies the properties $\mathcal{F}(v,0) \geq 0$ and $\mathcal{F}(v_{1}-v_{2},F(v_{1})-F(v_{2})) \geq 0$ for any $v,v_{1},v_{2} \in \mathbb{H}$. For the Hermitian extension of $\mathcal{F}$ for $v=(u,w) \in \mathbb{H}^{\mathbb{C}}$ and $\xi=(\xi_{1},\xi_{2}) \in \Xi^{\mathbb{C}}$ we have
\begin{equation}
	\label{EQ: ParabHermitianExtExample}
	\begin{split}
		\mathcal{F}^{\mathbb{C}}(u,w,\xi_{1},\xi_{2}) = \tau_{1}\int_{0}^{1}\operatorname{Re}\left[(\mu_{2}u(x)-\xi_{1}(x))^{*}(\xi_{1}(x)+\mu_{1}u(x) )\right ]dx +\\+ \tau_{2}\int_{0}^{1}\operatorname{Re}\left[(\varkappa_{2}u(x)-\xi_{2}(x))^{*}(\xi_{2}(x)+\varkappa_{1}u(x) ) \right]dx.
	\end{split}
\end{equation}
Note that for the case $\mu_{1}=\mu_{2} = \Lambda_{1}$, $\varkappa_{1}=\varkappa_{2}=\Lambda_{2}$ and $\tau_{1}=\tau_{2}=1$ the corresponding to \eqref{EQ: ParabHermitianExtExample} frequency inequality from \eqref{EQ: FreqConditionFormGeneral} takes the usual form as
\begin{equation}
	\label{EQ: ParabExampleGapParticular}
	|W(p)| < (\Lambda^{2}_{1} + \Lambda^{2}_{2})^{-1} = \Lambda^{-1} \text{ for all } p=-\nu + i\omega, \text{where } \omega \in \mathbb{R}.
\end{equation}
Here $\Lambda$ is the Lipschitz constant of $F$ and $\nu$ is fixed. For the general case the obtained region in the space of parameters $(\mu_{1},\mu_{2},\varkappa_{1},\varkappa_{2})$ will be larger than the one determined by \eqref{EQ: ParabExampleGapParticular}.

Let us write $\mathcal{F}$ from \eqref{EQ: ParabFormExample} as $\mathcal{F}(v,\xi)=\tau_{1} \mathcal{F}_{1}(v,\xi_{1}) + \tau_{2}\mathcal{F}_{2}(v,\xi_{2})$, where $\mathcal{F}_{1}$ and $\mathcal{F}_{2}$ are given by the corresponding integrals. In fact, the class of systems \eqref{EQ: ParabNeumannExample1}, with the above imposed assumptions on $f_{1}$ and $f_{2}$, is determined by the pair of forms $\mathcal{F}_{1}, \mathcal{F}_{2}$, which distinguish two different inputs $\xi_{1}$ and $\xi_{2}$ in the control system, and the class of systems determined by the form $\mathcal{F}$ is larger than this class. So, the frequency condition corresponding to $\mathcal{F}$ may be not optimal in the class determined by the pair $\mathcal{F}_{1}$ and $\mathcal{F}_{2}$, although this condition is better than the conditions that do not take into account this structure of the nonlinearity $F$.
\section*{Acknowledgments}
The author is thankful to the anonymous referee whose advices helped to improve the overall presentation of the work.
\section*{Funding}
The reported study was funded by RFBR according to the research project \textnumero~20-31-90008; by a grant in the subsidies form from the federal budget for the creation and development of international world-class math centers, agreement between MES RF and PDMI RAS No. 075-15-2019-1620; by V.~A.~Rokhlin grant for young mathematicians of St.~Petersburg.

\bibliographystyle{amsplain}

\begin{thebibliography}{10}
	
\bibitem{Anikushin2021DiffJ}
Anikushin M. M. Almost automorphic dynamics in almost periodic cocycles with one-dimensional inertial manifolds. \textit{Differencialnie Uravnenia i Protsesy Upravlenia}, 2 (2021) [in Russian].

\bibitem{Anikushin2020Geom}
Anikushin M.~M. Geometric theory of inertial manifolds for compact cocycles in Banach spaces. \textit{arXiv preprint}, arXiv:2012.03821v2 (2021).

\bibitem{Anikushin2020FreqDelay}
Anikushin M. M. Frequency theorem for the regulator problem with unbounded cost functional and its applications to nonlinear delay equations, \textit{arXiv preprint}, arXiv:2003.12499v4 (2021).

\bibitem{Anikushin2020Semigroups}
Anikushin M.~M. Nonlinear semigroups for delay equations in Hilbert spaces, inertial manifolds and dimension estimates, \textit{arXiv preprint}, arXiv:2004.13141v4 (2021).

\bibitem{Anikushin2020Red}
Anikushin M. M. A non-local reduction principle for cocycles in Hilbert spaces. \textit{J. Differ.
Equations}, \textbf{269}(9), 6699--6731 (2020).

\bibitem{Anikushin2019+OnCom} 
Anikushin M. M. On the compactness of solutions to certain operator inequalities arising from the Likhtarnikov-Yakubovich frequency theorem. \textit{Vestnik of Saint Petersburg University. Mathematics. Mechanics. Astronomy}, \textbf{53}(4) (2020) [in Russian].
	
\bibitem{Anik2020PB}
Anikushin M.~M. The Poincar\'{e}-Bendixson theory for certain compact semi-flows in Banach spaces. \textit{arXiv preprint} arXiv:2001.08627v3 (2020).

\bibitem{BoikoKuzetall2020}
Boiko I. M., Kuznetsov N. V., Mokaev R. N., Mokaev T. N., Yuldashev M. V., Yuldashev R. V. On counter-examples to Aizerman and Kalman conjectures. \textit{International Journal of Control}, 1--8 (2020).

\bibitem{ChepyzhovKostiankoZelik2019}
Chepyzhov V. V., Kostianko A., Zelik S. Inertial manifolds for the hyperbolic relaxation of semilinear parabolic equations. \textit{Discrete \& Continuous Dynamical Systems - B}, \textbf{24}(3), 1115--1142 (2019).

\bibitem{ChowLuSell1992}
Chow S.-N., Lu K., Sell G. R. Smoothness of inertial manifolds. \textit{J. Math. Anal. Applic.}, \textbf{169}(1), 283--312 (1992).

\bibitem{Chueshov2015}
Chueshov I. \textit{Dynamics of Quasi-stable Dissipative Systems}. Berlin: Springer (2015).

\bibitem{Datko1970}
Datko R. Extending a theorem of A. M. Liapunov to Hilbert space. \textit{J. Math. Anal. Appl.}, \textbf{32}(3) 610--616 (1970).

\bibitem{EngelNagel2000}
Engel K.-J., Nagel R. \textit{One-Parameter Semigroups for Linear Evolution Equations}. Springer-Verlag (2000).

\bibitem{Fabbri2003FreqTh}
Fabbri R., Johnson R., N\'{u}\~{n}ez C. On the Yakubovich Frequency Theorem for linear non-autonomous control processes. \textit{Discrete \& Continuous Dynamical Systems-A}, \textbf{9}(3), 677--704 (2003).

\bibitem{FoiasSellTemam1988}
Foias C., Sell G. R., Temam R. Inertial manifolds for nonlinear evolutionary equations. \textit{J. Differ. Equations}, \textbf{73}(2), 309--353 (1988).

\bibitem{Gelig1978}
Gelig A. Kh., Leonov G. A., Yakubovich V. A. \textit{Stability of Nonlinear Systems with Non-Unique Equilibrium State}. Nauka, Moscow (1978).

\bibitem{Henry1981}
Henry D. \textit{Geometric Theory of Semilinear Parabolic Equations}. Springer-Verlag (1981).

\bibitem{KokschSiegmund2002}
Koksch N., Siegmund S. Pullback attracting inertial manifolds for nonautonomous dynamical systems. \textit{J. Dyn. Differ. Equ.}, \textbf{14}(4), 889--941 (2002).

\bibitem{KostiankoZelikSA2020}
Kostianko A., Li X., Sun C., Zelik S. Inertial manifolds via spatial averaging revisited, \textit{arXiv preprint} arXiv:2006.15663 (2020).

\bibitem{KostiankoZelik2019Kwak}
Kostianko A., Zelik S. Kwak transform and inertial manifolds revisited. \textit{J. Dyn. Differ. Equ.}, 1--21 (2021).

\bibitem{Krein1971}
Krein S. G. \textit{Linear Differential Equations in Banach Space}, AMS, (1971).

\bibitem{KuzReit2020}
Kuznetsov N. V., Reitmann V. \textit{Attractor Dimension Estimates for Dynamical Systems: Theory and Computation}. Switzerland: Springer International Publishing AG (2021).

\bibitem{LeoKuz2013Hidden}
Leonov G. A., Kuznetsov N. V. Hidden attractors in dynamical systems. From hidden oscillations in Hilbert–Kolmogorov, Aizerman, and Kalman problems to hidden chaotic attractor in Chua circuits. \textit{Int. J. Bifurcat. Chaos}, \textbf{23}(1) (2013).

\bibitem{Likhtarnikov1977}
Likhtarnikov A. L., Yakubovich V. A. The Frequency Theorem for continuous one-parameter semigroups. \textit{Izv. Akad. Nauk SSSR Ser. Mat.} \textbf{41}(4), 895--911 (1977) [in Russian].

\bibitem{Likhtarnikov1976}
Likhtarnikov A. L., Yakubovich V. A. the Frequency Theorem for equations of evolutionary type. \textit{Sib. Math. J.}, \textbf{17}(5), 790--803 (1976).

\bibitem{LouisWexler1991}
Louis J.-Cl., Wexler D. The Hilbert space regulator problem and operator Riccati equation under stabilizability. \textit{Annales de la Soci\'{e}t\'{e} Scientifique de Bruxelles}, \textbf{105}(4), 137--165 (1991).

\bibitem{MalletParetSell1988IM}
Mallet-Paret J., Sell G. R. Inertial manifolds for reaction diffusion equations in higher space dimensions. \textit{J. Amer. Math. Soc.}, \textbf{1}(4), 805--866 (1988).

\bibitem{Miclavcic1991}
Miklav\v{c}i\v{c} M. A sharp condition for existence of an inertial manifold. \textit{J. Dyn. Differ. Equ.}, \textbf{3}(3), 437--456 (1991).

\bibitem{Popov1961}
Popov V. M. On absolute stability of non-linear automatic
control systems. \textit{Avtomat. i Telemekh.}, \textbf{22}(8), 961--979 (1961) [in Russian].

\bibitem{Proskurnikov2015}
Proskurnikov A. V. A new extension of the infinite-dimensional KYP lemma in the coercive case, \textit{IFAC-PapersOnLine}, \textbf{48}(1), 246--251 (2015).

\bibitem{Romanov1994}
Romanov A.~V. Sharp estimates of the dimension of inertial manifolds for nonlinear parabolic equations. Izvestiya: Mathematics, \textbf{43}(1), 31--47 (1994).

\bibitem{RosaTemam1996}
Rosa R.,  Temam R. Inertial manifolds and normal hyperbolicity. \textit{Acta Applicandae Mathematica}, \textbf{45}(1), 1--50 (1996).

\bibitem{SellYou2002}
Sell G. R., You Y. \textit{Dynamics of Evolutionary Equations}, Springer Science \& Business Media (2002).

\bibitem{SellYou1992}
Sell G. R., You Y. Inertial manifolds: the non-self-adjoint case. \textit{J. Differ. Equations}, \textbf{96}(2), 203--255 (1992).

\bibitem{Smith1994PB2}
Smith R. A. Orbital stability and inertial manifolds for certain reaction diffusion systems. \textit{P. Lond. Math. Soc.}, \textbf{3}(1), 91--120 (1994).

\bibitem{Smith1994PB1}
Smith R. A. Poincar\'{e}–Bendixson theory for certain reaction–diffusion boundary-value problems. \textit{Proc. Roy. Soc. Edinburgh Sect. A}, \textbf{124}(1), 33--69 (1994).

\bibitem{Yakubovich1973Sproc}
Yakubovich V. A. Minimization of quadratic functionals under quadratic constraints and the necessity of a frequency condition in the quadratic criterion for absolute stability of nonlinear control systems. \textit{Dokl. Akad. Nauk SSSR}, \textbf{209}(5), 1039--1042 (1973) [in Russian].

\bibitem{Zelik2014}
Zelik S. Inertial manifolds and finite-dimensional reduction for dissipative PDEs. \textit{P. Roy. Soc. Edinb. A}, \textbf{144}(6), 1245--1327 (2014).

\end{thebibliography}

\end{document}